%% file: 03_gamma_convergence_2.tex
\documentclass[a4paper, twoside]{scrartcl}

\usepackage{amsfonts,amssymb,amsmath,amsthm,amstext,amssymb,mathtools}
\usepackage{caption,subcaption,float,color,xcolor,graphicx,enumerate}
\usepackage{dsfont} 
\usepackage[normalem]{ulem} 
\usepackage[authoryear,sort,round]{natbib} 

\usepackage{tabularx,multirow}

\usepackage{hyperref}
\usepackage[noabbrev,capitalise,nosort,nameinlink]{cleveref}
\crefname{assumption}{Assumption}{Assumptions}

\usepackage{tikz-cd} 
\usepackage{adjustbox}
\usepackage{dashbox}

\definecolor{verylightgray}{rgb}{0.85,0.85,0.85}

\definecolor{boxred}{rgb}{0.7,0,0}

\usepackage{etex}
\usepackage[a4paper, top=88pt, bottom=88pt, left=72pt, right=72pt, headsep=16pt, footskip=28pt]{geometry}
\usepackage{hyperref}
\hypersetup{
	colorlinks,
	linkcolor=blue,
	citecolor=blue,
	urlcolor=blue,
}
\newcommand*{\arXiv}[1]{\bgroup\color{blue}\href{https://arxiv.org/abs/#1}{arXiv:#1}\egroup}
\newcommand*{\doi}[1]{\bgroup\color{blue}\href{https://doi.org/#1}{doi:#1}\egroup}
\newcommand*{\email}[1]{\bgroup\color{blue}\href{mailto:#1}{#1}\egroup}
\renewcommand*{\url}[1]{\bgroup\color{blue}\href{#1}{#1}\egroup}
\usepackage{enumitem, moreenum}
\setlist[enumerate]{nosep}
\setlist[itemize]{nosep}
\usepackage{mleftright} \mleftright
\renewcommand{\qedsymbol}{$\blacksquare$}
\renewenvironment{proof}[1][\proofname]{\noindent{\bfseries\sffamily #1.} }{\hfill\qedsymbol\medskip}
\usepackage[labelfont={sf,bf}]{caption}
\DeclareCaptionLabelSeparator{figlabelsep}{\,\,\,}
\usepackage{scrlayer-scrpage, xhfill}
\automark[section]{section}
\setkomafont{pageheadfoot}{\normalcolor\sffamily}
\setkomafont{pagenumber}{\normalfont\normalsize\sffamily}
\clearpairofpagestyles
\let\oldtitle\title
\renewcommand{\title}[1]{\oldtitle{#1}\newcommand{\theshorttitle}{#1}}
\newcommand{\shorttitle}[1]{\renewcommand{\theshorttitle}{#1}}
\let\oldauthor\author
\renewcommand{\author}[1]{\oldauthor{#1}\newcommand{\theshortauthor}{#1}}
\newcommand{\shortauthor}[1]{\renewcommand{\theshortauthor}{#1}}
\cohead{\xrfill[0.525ex]{0.6pt}~\theshorttitle~\xrfill[0.525ex]{0.6pt}}
\cehead{\xrfill[0.525ex]{0.6pt}~\theshortauthor~\xrfill[0.525ex]{0.6pt}}
\newcommand{\theabstract}[1]{\par\bgroup\noindent\textbf{\textsf{Abstract.}} #1\egroup}
\newcommand{\thekeywords}[1]{\par\smallskip\bgroup\noindent\textbf{\textsf{Keywords.}}\newcommand{\and}{ $\bullet$ } #1\egroup}
\newcommand{\themsc}[1]{\par\smallskip\bgroup\noindent\textbf{\textsf{2020 Mathematics Subject Classification.}}\newcommand{\and}{ $\bullet$ } #1\egroup}
\newcommand*{\affilref}[1]{\ref{affiliation#1}}
\newcommand*{\affiliation}[3]{
	\footnotetext[#1]{\label{affiliation#2} #3}
}

\numberwithin{equation}{section}
\numberwithin{figure}{section}
\numberwithin{table}{section}

\input{./newcommands.tex}
\renewcommand{\change}[1]{#1}
\usepackage{censor}

\title{\texorpdfstring{$\boldsymbol{\mathsf{\Gamma}}$}{Gamma}-convergence of Onsager--Machlup functionals}
\subtitle{Part II: Infinite product measures on Banach spaces}
\shorttitle{$\mathsf{\Gamma}$-convergence of Onsager--Machlup functionals: Part II}
\author{%
	Birzhan~Ayanbayev\textsuperscript{\affilref{Warwick}}%
	\and
	Ilja~Klebanov\textsuperscript{\affilref{FUB}}%
	\and
	Han Cheng Lie\textsuperscript{\affilref{Potsdam}}%
	\and
	T.~J.~Sullivan\textsuperscript{\affilref{Warwick}}%
}
\shortauthor{B.~Ayanbayev, I.~Klebanov, H.~C.~Lie, and T.~J.~Sullivan}

\begin{document}

\maketitle

\affiliation{1}{Warwick}{Mathematics Institute and School of Engineering, University of Warwick, Coventry, CV4 7AL, United Kingdom (\email{birzhan.ayanbayev@warwick.ac.uk}, \email{t.j.sullivan@warwick.ac.uk})}
\affiliation{2}{FUB}{Freie Universit{\"a}t Berlin, Arnimallee 6, 14195 Berlin, Germany (\email{klebanov@zedat.fu-berlin.de})}
\affiliation{3}{Potsdam}{Institut f\"ur Mathematik, Universit\"at Potsdam, Campus Golm, Haus 9, Karl-Liebknecht-Stra{\ss}e 24--25, Potsdam OT Golm 14476, Germany (\email{hanlie@uni-potsdam.de})}

\begin{abstract}
	\theabstract{\input{./bit-abstract.tex}}
	\thekeywords{\input{./bit-keywords.tex}}
	\themsc{\input{./bit-msc.tex}}
\end{abstract}

\input{./sec-01-introduction.tex}
\input{./sec-02-related.tex}
\input{./sec-03-notation.tex}
\input{./sec-04-convergence.tex}

\input{./sec-05-closing.tex}

\appendix

\input{./app-01-equivalence.tex}
\input{./app-02-technical.tex}

\setcounter{section}{17}

\input{./bit-acknowledgements.tex}

\bibliographystyle{abbrvnat}
\bibliography{references}
\addcontentsline{toc}{section}{References}

\end{document}

%% file: newcommands.tex
\newtheorem{theorem}{\sffamily Theorem}[section]
\newtheorem{corollary}[theorem]{\sffamily Corollary}
\newtheorem{proposition}[theorem]{\sffamily Proposition}
\newtheorem{lemma}[theorem]{\sffamily Lemma}
\theoremstyle{definition}
\newtheorem{definition}[theorem]{\sffamily Definition}
\newtheorem{remark}[theorem]{\sffamily Remark}
\newtheorem{example}[theorem]{\sffamily Example}
\newtheorem{conjecture}[theorem]{\sffamily Conjecture}

\newtheorem{assumption}[theorem]{\sffamily Assumption}
\newtheorem{notation}[theorem]{\sffamily Notation}

\crefname{conjecture}{Conjecture}{Conjectures}
\crefname{question}{Question}{Questions}

\newcommand*{\defeq}{\coloneqq}

\newcommand*{\defterm}{\textbf}

\DeclareMathOperator{\range}{\mathrm{ran}}
\DeclareMathOperator{\supp}{\mathrm{supp}}
\DeclareMathOperator{\spn}{\mathrm{span}}

\renewcommand{\emptyset}{\varnothing}
\renewcommand{\geq}{\geqslant}
\renewcommand{\leq}{\leqslant}

\newcommand*{\Gammalim}{\mathop{\ensuremath{\Gamma}\text{-}\mathrm{lim}}}

\newcommand*{\Naturals}{\mathbb{N}}
\newcommand*{\one}{\mathds{1}}
\newcommand*{\quark}{\setbox0\hbox{$x$}\hbox to\wd0{\hss$\cdot$\hss}}
\newcommand*{\vdotsquark}{\setbox0\hbox{$xx$}\hbox to\wd0{\hss$\vdots$\hss}}

\newcommand*{\Reals}{\mathbb{R}}
\newcommand*{\eReals}{\overline{\Reals}}
\newcommand*{\bP}{\mathbb{P}}
\newcommand*{\bE}{\mathbb{E}}

\newcommand*{\rd}{\mathrm{d}}
\newcommand*{\Torus}{\mathbb{T}}
\newcommand*{\ud}{\, \rd}

\newcommand*{\cJ}{\mathcal{J}}

\newcommand*{\logden}{\mathfrak{q}}
\newcommand*{\Id}{\textup{Id}}

\newcommand*{\Ball}[2]{\bgroup \color{green} B_{#2}(#1) \egroup}
\newcommand*{\cBall}[2]{B_{#2}(#1)}

\newcommand*{\BallExtra}[3]{\bgroup \color{green} B_{#2}^{#3}(#1) \egroup}
\newcommand*{\cBallExtra}[3]{\bar{B}_{#2}^{#3}(#1)}

\newcommand*{\cylBall}[3]{B_{#2}^{#3}(#1)}

\newcommand*{\rv}[1]{\boldsymbol{#1}}

\newcommand*{\prob}[1]{\mathcal{P}(#1)}

\newcommand*{\Cauchy}{\mathcal{C}}

\newcommand{\Borel}[1]{\mathcal{B}(#1)}

\newcommand{\absval}[1]{\lvert #1 \rvert}

\newcommand{\norm}[1]{\lVert #1 \rVert}

\newcommand{\set}[2]{\{ #1 \mid #2 \}}

\newcommand{\bignorm}[1]{\bigl\Vert #1 \bigr\Vert}

\newcommand{\Absval}[1]{\left\vert #1 \right\vert}

\newcommand{\Norm}[1]{\left\Vert #1 \right\Vert}
\newcommand{\Set}[2]{\left\{ #1 \,\middle\vert\, #2 \right\}}

\newcommand{\todo}[1]{\bgroup\color{red}TODO: #1\egroup}
\newcommand{\ba}[1]{\bgroup\color{olive}Birzhan: #1\egroup}
\newcommand{\tjs}[1]{\bgroup\color{cyan}Tim: #1\egroup}
\newcommand{\hcl}[1]{\bgroup\color{brown}Han: #1\egroup}
\newcommand{\ik}[1]{\bgroup\color{violet}Ilja: #1\egroup}

\newcommand{\change}[1]{\bgroup\color{olive}#1\egroup}

\usepackage{acro}
\DeclareAcronym{BIP}{short=BIP, long=Bayesian inverse problem}
\DeclareAcronym{DFG}{short=DFG, long=Deutsche Forschungs\-gemein\-schaft}
\DeclareAcronym{MAP}{short=MAP, long=maximum a posteriori}
\DeclareAcronym{OM}{short=OM, long=Onsager--Machlup}

%% file: bit-abstract.tex
We derive Onsager--Machlup functionals for countable product measures on weighted $\ell^{p}$ subspaces of the sequence space $\Reals^{\Naturals}$.
Each measure in the product is a shifted and scaled copy of a reference probability measure on $\Reals$ that admits a sufficiently regular Lebesgue density.
We study the equicoercivity and $\Gamma$-convergence of sequences of Onsager--Machlup functionals associated to convergent sequences of measures within this class.
We use these results to establish analogous results for probability measures on separable Banach or Hilbert spaces, including Gaussian, Cauchy, and Besov measures with summability parameter $1 \leq p \leq 2$.
Together with Part~I of this paper, this provides a basis for analysis of the convergence of maximum a posteriori estimators in Bayesian inverse problems and most likely paths in transition path theory.

%% file: bit-keywords.tex
Bayesian inverse problems%
\and%
$\Gamma$-convergence%
\and%
maximum a posteriori estimation%
\and%
Onsager--Machlup functional%
\and%
small ball probabilities%
\and%
transition path theory%

%% file: bit-msc.tex
49Q20
\and
60B11
\and
49J45
\and
49K40
\and
62F15

%% file: sec-01-introduction.tex
\section{Introduction}
\label{section:Introduction}

A \ac{MAP} estimator is an important feature of a \ac{BIP} because of its interpretation as a \emph{mode} of the posterior distribution, i.e.\ as a point in parameter space $X$ to which the posterior assigns the most mass, relative to other points.
This interpretation is only heuristic, because even in the straightforward case that the parameter space has finite dimension and the posterior admits a Lebesgue density, every point will have measure zero.
To make the interpretation rigorous, one can consider --- for a given probability measure $\mu$ on $X$ --- the behaviour of ratios of small ball probabilities $\tfrac{ \mu ( \cBall{x_{1}}{r} ) }{ \mu ( \cBall{x_{2}}{r} ) }$ for infinitesimally small $r$ and for any two parameters $x_{1}, x_{2} \in X$.
Intuitively, if $x_{2}$ is a mode of $\mu$, then, for any $x_{1}$, the limit superior of this ratio must be less than or equal to 1.

In Part~I of this paper \citep{AyanbayevKlebanovLieSullivan2021_I}, we called any $x_{2}$ that satisfies the limit superior inequality in the previous paragraph a \emph{global weak mode} of $\mu$, and showed that, under certain assumptions, a point is a global weak mode if and only if it minimises an \ac{OM} functional $I_{\mu} \colon X \to \eReals$ of $\mu$.
In practice, the full posterior is not accessible and must be approximated, and we also analysed the convergence behaviour of the modes associated to an arbitrary collection $\set{ \mu^{(n)} }{ n \in \Naturals \cup \{ \infty \} }$ of measures defined on a metric space $X$, where $\mu^{(\infty)}$ plays the role of the full posterior and $(\mu^{(n)})_{n\in\Naturals}$ plays the role of a sequence of approximate posteriors.
Our findings were as follows:
\begin{enumerate}[label = (\alph*)]
	\item
	\label{item:Correspondence_OMminimisers_modes}
	If (extended) \ac{OM} functionals $I_{\mu^{(n)}} \colon X \to \eReals$ exist for each $n \in \Naturals \cup \{ \infty \}$ and $(I_{\mu^{(n)}})_{n \in \Naturals}$ is an equicoercive sequence with $\Gammalim_{n \to \infty} I_{\mu^{(n)}} = I_{\mu^{(\infty)}}$, then minimisers of $I_{\mu^{(n)}}$ converge (up to taking subsequences) to a minimiser of $I_{\mu^{(\infty)}}$ \citep[Section~4]{AyanbayevKlebanovLieSullivan2021_I}. 
	\item 
	Since modes of $\mu^{(n)}$ are minimisers of their \ac{OM} functionals, it follows that modes converge (up to taking subsequences) to a mode of $\mu^{(\infty)}$ \citep[Section~4]{AyanbayevKlebanovLieSullivan2021_I}. 
	\item
	\label{item:posterior_properties_follow_from_prior_properties}
	Suppose that the measures $\mu^{(n)}$, $n \in \Naturals \cup \{ \infty \}$, are posteriors given by Radon--Nikodym derivatives (cf.\ \citealt{Stuart2010})
	\[
		\frac{\rd \mu^{(n)}}{\rd \mu_{0}^{(n)}}
		\propto
		\exp(-\Phi^{(n)}),
	\]
	where $\Phi^{(n)} \colon X \to \Reals$ are the potentials (negative log-likelihoods) and $\mu_{0}^{(n)}$ are the priors, $n \in \Naturals \cup \{ \infty \}$.
	Under rather weak assumptions on the $\Phi^{(n)}$, if the conditions in \ref{item:Correspondence_OMminimisers_modes} hold for the priors, then they also hold for the posteriors.
	In particular, the existence of the \ac{OM} functionals $I_{\mu^{(n)}}$ for the posteriors follows from the existence of the \ac{OM} functionals for the priors \citep[Section~6]{AyanbayevKlebanovLieSullivan2021_I}. 
\end{enumerate}

In principle, establishing $\Gamma$-convergence and equicoercivity would require explicit formulae for the \ac{OM} functionals of the posteriors, and such formulae can be difficult to obtain.
Fortunately, by \ref{item:posterior_properties_follow_from_prior_properties}, we only need to prove $\Gamma$-convergence and equicoercivity for the \ac{OM} functionals of the priors and continuous convergence of the potentials.
Indeed, for some commonly-used priors, the \ac{OM} functionals of the priors have a simple form and the requisite $\Gamma$-convergence and equicoercivity calculations can be performed more-or-less explicitly.

In Part~I of this paper \citep{AyanbayevKlebanovLieSullivan2021_I}, we determined \ac{OM} functionals and proved \ref{item:Correspondence_OMminimisers_modes} for possibly degenerate Gaussian measures, as well as for Besov-$1$ measures.
In this paper, we aim to do the same for a rather large class of countable product measures defined on weighted sequence spaces.
This class of measures consists of countable products of scaled and shifted copies of a reference probability measure $\mu_{0}$ on $\Reals$, where $\mu_{0}$ admits a sufficiently regular Lebesgue density.
The class includes Gaussian measures, Cauchy measures, and Besov-$p$ measures for $1 \leq p \leq 2$.
The precise description of this class is given in \Cref{assump:basic_assumptions_for_product_measures}.

The first main contribution of this paper, \Cref{thm:OM_for_product_measures}, shows the existence of and derives an explicit formula for \ac{OM} functionals of measures in this class under another technical assumption.
The second main contribution is to prove equicoercivity and $\Gamma$-convergence of \ac{OM} functionals associated to a convergent sequence in this class, where convergence is meant in the sense of convergence of the scale and shift sequences, and convergence of the Lebesgue densities of the reference probability measures:
see \Cref{thm:Equicoercivity_for_product_measures,thm:Gamma_convergence_for_product_measures}.
As concrete examples, we consider Besov-$p$ measures for $1 \leq p \leq 2$, and Cauchy measures.
Since Bayesian inference is often performed on infinite-dimensional separable Banach or Hilbert spaces, we also translate the results from the weighted sequence space setting to the separable Banach or Hilbert space setting.

The main challenge in this work is proving the existence of the extended \ac{OM} functionals.
In this paper, we consider two approaches for this.
The first approach, which we call the \emph{continuity approach}, considers shifted measures $\mu_{h}(\quark) \defeq \mu(\quark - h)$ and the corresponding Radon--Nikodym derivatives $r_{h}^{\mu} \defeq \frac{\rd \mu_{h}}{\rd \mu}$, whenever they exist.
The main idea of this approach, which has previously been used by \citet{HelinBurger2015} and \citet{AgapiouBurgerDashtiHelin2018}, is to consider the negative logarithm of the function $E\ni h\mapsto r_{-h}^{\mu}(u_{\ast})$, where $u_{\ast}$ is some suitable reference point, and  $E\subseteq X$ is a subset on which $r_{h}^{\mu}$ is continuous and may depend on the reference point $u_{\ast}$. 
We make some contributions to this approach.
Ultimately, we do not use it for the derivation of our main results, because proving continuity on a sufficiently large subset $E\subseteq X$ turns out to be more challenging than using a different approach.

The second approach, which we call the \emph{direct approach}, avoids considering continuity of $r_{-h}^{\mu}$, and directly addresses the limit of the ratio $\frac{ \mu ( \cBall{x_{1}}{r} ) }{ \mu ( \cBall{x_{2}}{r} ) }$ as $r \searrow 0$ to derive the \ac{OM} functional of $\mu$ on a sufficiently large subset $E\subseteq X$.
By removing the constraint on $E$ that $r_{-h}^{\mu}$ must be continuous on $E$, we can prove a formula for the \ac{OM} functional using this direct approach, for the class of probability measures mentioned above.

We emphasise, however, that in both approaches it is important to consider points in $X \setminus E$ with great care.
In the direct approach, we achieve this by proving a property $M(\mu,E)$ which guarantees that we do not miss any modes outside of $E$.

The structure of the paper is as follows.
In \Cref{sec:related} we discuss related work.
\Cref{sec:notation} introduces key notation and concepts, including the formal definition of the \ac{OM} functional.
In \Cref{sec:convergence_of_OM_functionals}, we present the main results of this paper, namely the derivation of \ac{OM} functionals of certain product measures on the sequence space $\Reals^{\Naturals}$ as well as the $\Gamma$-convergence and equicoercivity properties of sequences of such measures (and the images of such measures in Hilbert and Banach spaces).
In \Cref{sec:closing}, we summarise the results of the paper and suggest some directions for future work.
We collect auxiliary results in \Cref{sec:equivalence} and state technical proofs in \Cref{sec:technical}.

%% file: sec-02-related.tex
\section{Overview of related work}
\label{sec:related}

\Ac{OM} functionals have been extensively studied in the context of stochastic processes defined by stochastic differential equations;
see e.g.\ \cite[Chapter~7]{Ledoux1996} and the references therein.
However, $\Gamma$-convergence does not appear to have been considered in this context until the work of \citet{Pinski2012}.
In their work, $\Gamma$-convergence tools were used to study the minimisers of \ac{OM} functionals in the zero temperature limit.
\citet{Lu2017a} considered optimal Gaussian approximations of the law of a diffusion process with respect to the Kullback--Leibler divergence using $\Gamma$-convergence, and studied the relationship between the \ac{OM} functional and the so-called Freidlin--Wentzell rate functional.
Some examples of recent work that further investigate this relationship include \citep{Du2021,Li2021}.

\Ac{OM} functionals have only recently been studied in the context of \ac{BIP}s and their \ac{MAP} estimators, beginning with the seminal work of \citet{DashtiLawStuartVoss2013}, and continuing with \citep{HelinBurger2015,DunlopStuart2016,Clason2019GeneralizedMI}, for example.
The importance of the \ac{OM} functional in this context is that its minimisers are the modes (\ac{MAP} estimators) of the posterior measure.
However, these works establish \ac{OM} functionals only for very few measures and do not consider $\Gamma$-convergence, as they only study a single fixed posterior measure instead of a sequence of such measures.
As far as we are aware, the only application of $\Gamma$-convergence tools in the context of \ac{BIP}s appears to be the work of \citet{Lu2017b}, where, the goal is to find optimal Gaussian approximations of non-Gaussian probability measures on $\Reals^{d}$ with respect to the Kullback--Leiber divergence.
The $\Gamma$-limits of interest are specified in terms of increasing quantity of data or decreasing amplitude of noise in the data.
The $\Gamma$-limit is used to characterise frequentist consistency properties of the measure, including a Bernstein--von Mises result.
However, \citet{Lu2017b} do not mention \ac{OM} functionals.

%% file: sec-03-notation.tex
\section{Preliminaries and notation}
\label{sec:notation}

Throughout this article, $X$ will denote a topological space, which in many cases will be a metric, normed, Banach or Hilbert space.
When thought of as a measurable space, $X$ will be equipped with its Borel $\sigma$-algebra $\Borel{X}$, which is generated by the collection of all open sets.
If $X$ is a metric space, then we write $\cBall{x}{r}$ for the open ball in $X$ of radius $r$ centred on $x$, in which case $\Borel{X}$ is generated by the collection of all open balls.
The most prominent spaces considered in this manuscript are the real sequence spaces $\ell^{p} \defeq \ell^{p}(\Naturals)$ of $p$\textsuperscript{th}-power summable sequences, $1 \leq p < \infty$, as well as the $\alpha$-weighted $\ell^{p}$ spaces defined by
\begin{equation}
	\label{eq:gamma_weighted_sequence_space}
	\ell^{p}_{\alpha}
	\defeq
	\Set{ x \in \Reals^{\Naturals} }{ ( x_{k} / \alpha_{k} )_{k \in \Naturals} \in \ell^{p} },
	\qquad
	\norm{ x }_{\ell^{p}_{\alpha}}
	\defeq
	\bignorm{ ( x_{k} / \alpha_{k} )_{k \in \Naturals}}_{\ell^{p}},
\end{equation}
where $\alpha=(\alpha_{k})_{k \in \Naturals} \in \Reals_{> 0}^{\Naturals}$.
The $\ell^{p}$ and $\alpha$-weighted $\ell^{p}$ spaces are separable Banach spaces.

In many cases, we will first define the measure $\mu$ on $(\Reals^{\Naturals}, \Borel{\Reals^{\Naturals}})$, where $\Reals^{\Naturals}$ is equipped with the product topology, show that $\mu(X) = 1$ for $X = \ell_{\alpha}^{p}$ for some $1 \leq p < \infty$ and $\alpha \in \Reals_{> 0}^{\Naturals}$, and then view $\mu$ as a measure on $(X,\Borel{X})$.
For this purpose, it is important to note that the Borel $\sigma$-algebra $\Borel{X}$ is contained in the Borel $\sigma$-algebra $\Borel{\Reals^{\Naturals}}$;
see \Cref{lemma:Sigma_algebras_are_fine}.

The set of all probability measures on $(X, \Borel{X})$ will be denoted $\prob{X}$.
We denote its elements by $\mu$, $\nu$, $\mu_{0}$, $\mu^{(n)}$, $n\in\Naturals \cup \{ \infty \}$, etc.
The topological support of a measure $\mu \in \prob{X}$ on a metric space $X$ is
\begin{equation}
	\label{eq:support_of_measure}
	\supp (\mu) \defeq \set{ x \in X }{ \text{for all $r > 0$, } \mu ( \cBall{x}{r} ) > 0 } ,
\end{equation}
which is always a closed subset of $X$.

We write $\eReals$ for the extended real line $\Reals \cup \{ \pm \infty \}$, i.e.\ the two-point compactification of $\Reals$, and $\eReals_{\geq 0}\defeq \Reals_{\geq 0}\cup \{\infty\}$.
We denote the absolute continuity of $\mu$ with respect to $\nu$ by $\mu \ll \nu$, their equivalence (i.e.\ mutual absolute continuity) by $\mu \sim \nu$, and their mutual singularity by $\mu \perp \nu$.

As motivated in \Cref{section:Introduction}, we now introduce the term  ``Onsager--Machlup functional'' of a measure $\mu$, the minimisers of which correspond exactly to global weak modes of $\mu$ under certain assumptions \citep[Proposition~4.1]{AyanbayevKlebanovLieSullivan2021_I}. 

\begin{definition}
	\label{defn:Onsager--Machlup}
	Let $X$ be a metric space and let $\mu \in \prob{X}$.
	We say that $I = I_{\mu} = I_{\mu, E} \colon E \to \Reals$, with $E \subseteq \supp(\mu) \subseteq X$, is an \defterm{Onsager--Machlup functional} (\ac{OM} functional) for $\mu$ if
	\begin{equation}
		\label{eq:Onsager--Machlup}
		\lim_{r \searrow 0} \frac{ \mu ( \cBall{x_{1}}{r} ) }{ \mu ( \cBall{x_{2}}{r} ) } = \exp ( I(x_{2}) - I(x_{1}) )
		\text{ for all $x_{1}, x_{2} \in E$.}
	\end{equation}
	We say that \defterm{property $M(\mu, E)$} is satisfied if, for some $x^{\star} \in E$,
	\begin{equation}
		\label{eq:liminf_small_ball_prob}
		x \in X \setminus E \implies \lim_{r \searrow 0} \frac{ \mu ( \cBall{x}{r} ) }{ \mu ( \cBall{x^{\star}}{r} ) } = 0 ,
	\end{equation}
	and in this situation we extend $I$ to a function $I \colon X \to \eReals$ with $I(x) \defeq + \infty$ for $x \in X \setminus E$.
\end{definition}

As we remark in Part~I of this paper \citep[Section~3]{AyanbayevKlebanovLieSullivan2021_I}, property $M(\mu,E)$ does not depend on the choice of $x^{\star}$ in \eqref{eq:liminf_small_ball_prob}.
The importance of property $M(\mu, E)$ is that it guarantees that we only need to look for global weak modes of $\mu$ within $E$ and may freely ignore points in $X \setminus E$.
This also justifies setting $I \defeq +\infty$ outside $E$.
However, in order for this property to hold, the subset $E$ on which an \ac{OM} functional can be defined needs to be chosen to be as large as possible.
On the other hand, any measure has an \ac{OM} functional on sufficiently small $E$ (such as a singleton set), and so there is a certain tension between existence of an \ac{OM} functional and the $M$-property.
We recall also that \ac{OM} functionals are at best unique up to the addition of real constants \citep[Remark 3.4]{AyanbayevKlebanovLieSullivan2021_I}.
Whenever we prove $\Gamma$-convergence and equicoercivity, we use the same version of the \ac{OM} functional.

The following terminology will be necessary for the continuity approach mentioned in \Cref{section:Introduction}.

\begin{definition}
	\label{def:quasi_invariance_bogachev}
	When $X$ is a linear topological space, $\mu \in \prob{X}$, and $h \in X$, we write $\mu_{h}$ for the \defterm{shifted measure}
	\begin{equation}
		\label{eq:shifted_measure}
		\mu_{h}(A) \defeq \mu(A - h) = \mu ( \set{ a - h }{ a \in A } )
		\quad
		\text{for each $A \in \Borel{X}$.}
	\end{equation}
	That is, $\mu_{h}$ is the push-forward of $\mu$ via the translation map $x \mapsto x + h$.
	The measure $\mu$ is called \defterm{quasi-invariant along $h$}, if, for all $t \in \Reals$, $\mu_{th} \sim \mu$.
	We define
	\begin{equation}
		\label{eq:Q_space}
		Q(\mu) \defeq \set{ h \in X }{ \mu \text{ is quasi-invariant along } h }.
	\end{equation}
	For $h\in Q(\mu)$, we define the \defterm{shift density} $r^\mu_h \defeq \frac{\rd \mu_{h}}{\rd \mu} \in L^1(\mu)$ as the Radon--Nikodym derivative of $\mu_{h}$ with respect to $\mu$, i.e.
	\begin{equation}
		\label{eq:r_mu_h_radon_nikodym_derivative_of_mu_h_wrt_mu}
		\mu_{h}(A) = \int_{A} r^\mu_h(x) \, \mu(\rd x)
		\quad
		\text{for each $A \in \Borel{X}$.}
	\end{equation}
\end{definition}

\change{
\begin{remark}
\label{remark:shift_properties_metric_independent}
Note that, in contrast to \ac{OM} functionals, the shift-quasi-invariance space $Q(\mu)$ and the shift density $r_{h}^{\mu}$ do not depend on a particular metric.
\end{remark}
}

%% file: sec-04-convergence.tex
\section{\ac{OM} functionals for product measures; equicoercivity and \texorpdfstring{$\boldsymbol{\mathsf{\Gamma}}$}{Gamma}-convergence}
\label{sec:convergence_of_OM_functionals}

Determining the shift-quasi-invariance space $Q(\mu)$, the shift density $r_{h}^{\mu}$ and the \ac{OM} functional $I_{\mu}$ for a general measure $\mu$ on an infinite-dimensional space is a challenging task, as is establishing $\Gamma$-convergence and equicoercivity for such \ac{OM} functionals.
In the following, we describe two approaches that apply to a class of shifted product measures $\mu = \bigotimes_{k\in\Naturals} \mu_{k}$, $\mu_{k}(\quark) \defeq \mu_{0}(\gamma_{k}^{-1} (\quark - m_{k}))$.
This class includes many of the classical prior measures that arise in the study of inverse problems, such as Gaussian, Besov, and Cauchy measures.
Their common structure is summarised by the following assumptions on $\mu$, where \ref{item:basic_assumption_product_full_measure}--\ref{item:basic_assumption_product_scaled_marginals} should be seen as common basic assumptions,
while \ref{item:basic_assumption_product_Shepp_condition}--\ref{item:basic_assumption_product_Besov_p_1_2} are technical assumptions that will be used individually in specific settings.

\begin{assumption}
	\label{assump:basic_assumptions_for_product_measures}
	We introduce the following assumptions on the countable product measure $\mu \defeq \bigotimes_{k\in\Naturals} \mu_{k} \in \prob{\Reals^{\Naturals}}$:
	\begin{enumerate}[label = (A\arabic*)]
		\item
		\label{item:basic_assumption_product_full_measure}
		\change{Support in $\ell_{\alpha}^{p}$:}
		$\mu(X) = 1$ where $(X,\norm{\quark}_{X}) = (\ell_{\alpha}^{p},\norm{\quark}_{\ell_{\alpha}^{p}})$ for some $\alpha \in \Reals_{>0}^{\Naturals}$ and $1\leq p < \infty$.
		Consider $\mu$ as a measure on the Banach space $X$.
		\item
		\change{Continuous, symmetric reference density:}
		\label{item:basic_assumption_product_mu_0}
		$\mu_{0} \in \prob{\Reals}$ is a probability measure on $(\Reals, \Borel{\Reals})$ with continuous and symmetric Lebesgue probability density $\rho$ such that $\rho|_{\Reals_{\geq 0}}$ is strictly monotonically decreasing.
		\item
		\change{Affine change of variables:}
		\label{item:basic_assumption_product_scaled_marginals}
		$\mu_{k}(A) \defeq \mu_{0}(\gamma_{k}^{-1} (A - m_{k}))$, $A\in\Borel{\Reals}$, where $\gamma \in \Reals_{>0}^{\Naturals}$, $m \in X$.
		\item
		\change{Finite Fisher information:}
		\label{item:basic_assumption_product_Shepp_condition}
		$\rho$ is Lebesgue-a.e.\ positive, locally absolutely continuous and $\int_{\Reals} (\rho'(u))^{2} / \rho(u) \, \rd u < \infty$.
		\item
		\change{Smooth reference density:}
		\label{item:basic_assumption_product_integrable_second_derivative}
		$\rho \in C^{2}(\Reals)$ and $\rho'' \in L^{1}(\Reals)$.
		\item
		\change{Besov measure:}
		\label{item:basic_assumption_product_Besov_p_1_2}
		$\mu = B^{s}_{p}$ is a Besov measure with $1\leq p \leq 2$ and $\alpha = \delta$.
		For a definition of $B^{s}_{p}$ and $\delta$, see \Cref{sec:Gamma_Besov}.
	\end{enumerate}
\end{assumption}

\begin{remark}
	While many product measures satisfy 		\ref{item:basic_assumption_product_integrable_second_derivative}, the Besov measure $\mu = B^{s}_{p}$ with $1\leq p < 2$ does not have a sufficiently smooth probability density $\rho$.
	This is why we treat this case separately, via	\ref{item:basic_assumption_product_Besov_p_1_2}.
	
	\change{Note also that, since the shift-quasi-invariance space $Q(\mu)$ and the shift density $r_{h}^{\mu}$ do not depend on the particular metric (cf.\ \cref{remark:shift_properties_metric_independent}), the corresponding results hold on all of $\Reals^{\Naturals}$ and do not require \ref{item:basic_assumption_product_full_measure}.}
\end{remark}

Many prior measures of interest, such as Gaussian, Cauchy and Besov measures, are often defined on Banach or Hilbert spaces $Z$ that are not subspaces of $\Reals^{\Naturals}$.
Thus, we introduce the following notation, which will allow us to translate the results from $\ell_{\alpha}^{p} \subseteq \Reals^{\Naturals}$ to $Z$:

\begin{notation}
	\label{notation:From_RN_to_Banach}
	Let $X = \ell_{\alpha}^{p}$ for some $1 \leq p < \infty$ and $\alpha \in \Reals_{> 0}^{\Naturals}$.
	Let $Z$ denote a separable Banach space with Schauder basis $\psi = (\psi_{k})_{k\in\Naturals}$ such that the synthesis operator
	\begin{align*}
		S_{\psi} & \colon X \to Z,
		&
		x = (x_{k})_{k \in \Naturals} & \mapsto \sum_{k \in \Naturals} x_{k}\psi_{k}, \\
		\intertext{and the coordinate operator}
		T_{\psi} & \colon Z \to \Reals^{\Naturals},
		&
		z = \sum_{k \in \Naturals} v_{k} \psi_{k} & \mapsto (v_{k})_{k\in\Naturals},
	\end{align*}
	are well defined and $S_{\psi}$ is a continuous embedding.
	Note that $T_{\psi} \circ S_{\psi} = \Id_{X}$.
	For a probability measure $\mu \in \prob{X}$, we denote by $\mu_{\psi} \defeq (S_{\psi})_{\#} \mu$ the push-forward of $\mu$ under $S_{\psi}$.
	If instead of $\mu \in \prob{X}$ we have $\mu \in \prob{\Reals^{\Naturals}}$ and $\mu(X) = 1$, then $\mu_{\psi}$ denotes the push-forward of the restriction of $\mu$ to $(X,\Borel{X})$.
\end{notation}

\change{
\begin{example}
	The standard example of the setup described by \Cref{notation:From_RN_to_Banach} is to consider $(\psi_{k})_{k \in \Naturals}$ to be the standard Fourier basis of the space $Z = L^{2} (\mathbb{T}^{d}; \Reals)$ of square-integrable periodic functions in $d$ variables.
	Taking $p = 2$ and $\alpha = (1, 1, \dots)$, the operators $T_{\psi}$ and $S_{\psi}$ are isometries --- they are the Fourier transform and its inverse, respectively.
	By way of contrast, taking $\alpha_{n} \sim n^{s}$ for $s > 0$ yields a Sobolev space as $Z$, and further taking $p \neq 2$ yields a Besov space.
\end{example}
}

Most of our results on $X = \ell_{\alpha}^{p}$ can be transferred to the Banach space $Z$ via $S_{\psi}$.
However, for the statements concerning \ac{OM} functionals, we will assume in addition that $S_{\psi}$ is an isometry, i.e.\ that $\norm{x}_{X}=\norm{S_{\psi}x}_Z$ for every $x \in X$.
This is because the definition of the \ac{OM} functional depends strongly on the metric, and because even equivalent norms can yield different \ac{OM} functionals \citep[Example~B.4]{AyanbayevKlebanovLieSullivan2021_I}. 

\begin{lemma}
	\label{lemma:From_RN_to_Banach}
	Suppose that \Cref{assump:basic_assumptions_for_product_measures} \ref{item:basic_assumption_product_full_measure}--\ref{item:basic_assumption_product_scaled_marginals} hold.
	If $S_{\psi}$ in \Cref{notation:From_RN_to_Banach} is an isometry, then
	\begin{equation}
		\label{equ:Ball_masses_from_RN_to_Banach}
		\mu_{\psi}(\cBall{h}{r})
		=
		\begin{cases}
			\mu(\cBall{T_{\psi}h}{r}) & \text{if } h \in \range S_{\psi} = S_{\psi}(\ell_{\alpha}^{p}),
			\\
			0 & \text{otherwise.}
		\end{cases}
	\end{equation}
	Hence, if $I_{\mu} \colon X \to \eReals$ is an \ac{OM} functional for $\mu$, then
	\begin{equation}
		\label{equ:OM_from_RN_to_Banach}
		I_{\mu_{\psi}} \colon Z \to \eReals,
		\qquad
		I_{\mu_{\psi}} (h)
		=
		\begin{cases}
			I_{\mu}(T_{\psi}h) & \text{if } h \in \range S_{\psi},
			\\
			+\infty & \text{otherwise,}
		\end{cases}
	\end{equation}
	defines an \ac{OM} functional for $\mu_{\psi}$.
	Similarly, if $I_{\mu_{\psi}} \colon X \to \eReals$ is an \ac{OM} functional for $\mu_{\psi}$, then $I_{\mu} \defeq I_{\mu_{\psi}} \circ S_{\psi} \colon X \to \eReals$ defines an \ac{OM} functional for $\mu$.
\end{lemma}

\begin{proof}
	If $S_{\psi}$ from \Cref{notation:From_RN_to_Banach} is an isometry, then for any $h \in \range S_{\psi}$,
	\[
	\mu_{\psi}(\cBall{h}{r})
	=
	\mu(S_{\psi}^{-1}(\cBall{h}{r}))
	=
	\mu(\cBall{T_{\psi}h}{r}).
	\]
	Note that $\range S_{\psi}$ is complete and therefore closed in $Z$.
	Hence, for $h \notin \range S_{\psi} = S_{\psi}(\ell_{\alpha}^{p})$, there exists $r_{0} > 0$ such that $\cBall{h}{r_{0}} \cap \range S_{\psi} = \emptyset$.
	Thus, for any $0 < r < r_{0}$,
	\[
	\mu_{\psi}(\cBall{h}{r})
	=
	\mu(S_{\psi}^{-1}(\cBall{h}{r}))
	=
	\mu(\emptyset)
	=
	0,
	\]
	proving \eqref{equ:Ball_masses_from_RN_to_Banach}.
	The second-last and last statements follow from \Cref{defn:Onsager--Machlup} by choosing $E_{\psi} \defeq \{ S_{\psi}x \mid x\in X,\ I_{\mu}(x) < \infty \}$ and proving property $M(\mu_{\psi},E_{\psi})$ via \eqref{equ:Ball_masses_from_RN_to_Banach}, and by choosing $E \defeq \{ x\in X \mid I_{\mu_{\psi}}(S_{\psi}x) < \infty \}$ and proving property $M(\mu,E)$ via \eqref{equ:OM_from_RN_to_Banach} respectively.
\end{proof}

The two approaches that we consider for establishing \ac{OM} functionals consist of the \emph{continuity approach}, which we present in \cref{ssec:continuity_approach}, and the \emph{direct approach}, which we present in \cref{ssec:direct_approach}.
In the literature on \ac{MAP} estimators, the continuity approach appears to have been first proposed by \citet{HelinBurger2015}.
The approach connects the \ac{OM} functional for $\mu$ with the continuity of the shift density $r_{h}^{\mu}$ from \Cref{def:quasi_invariance_bogachev}.
In contrast, the direct approach considers the ratio of small ball probabilities directly, and does not require continuity of the shift density $r_{h}^{\mu}$.

\input{./sec-04-convergence-01.tex}

\input{./sec-04-convergence-03.tex}

\input{./sec-04-convergence-04.tex}

%% file: sec-04-convergence-01.tex
\subsection{Continuity approach}
\label{ssec:continuity_approach}

We present some results that are related to the approach from \citep{HelinBurger2015}, i.e.\ the approach of using continuity of the shift density $r_{h}^{\mu}$.
The results  \Cref{lem:continuous_on_support_representative_lemma,cor:continuous_on_support_representative_lemma_for_balls} do not require the product structure of the measure as formulated in \Cref{assump:basic_assumptions_for_product_measures}.
\Cref{thm:shift_quasi_invariance_space_for_product_measures} derives the shift-quasi-invariance spaces $Q(\mu)$ and shift densities\footnote{We wish to highlight the case of Besov-$p$ measures: In previous work \citep{AgapiouBurgerDashtiHelin2018}, formulas for $Q(\mu)$ and $r_{h}^{\mu}$ could only be derived for $p=1$ by a considerable amount of work, while our results include the cases $1\leq p < \infty$ and the proof is a rather simple application of \Cref{thm:kakutanis_theorem,thm:InvariantSpaceShepp}.} $r_{h}^{\mu}$ specifically for product measures fulfilling \Cref{assump:basic_assumptions_for_product_measures} \ref{item:basic_assumption_product_mu_0}--\ref{item:basic_assumption_product_Shepp_condition}. 
\change{These assumptions refer to the continuity and symmetry of the reference density $\rho$, the affine transformation relationship between the $\mu_{k}$ and $\mu_{0}$, and the finite Fisher information condition.}
One of the key disadvantages of this approach is that it requires the existence of representatives of shift densities or logarithmic derivatives that are continuous on sets of full measure, see e.g.\ \citep[Assumption (A1)]{HelinBurger2015}.
This is the reason why we do not use either \Cref{lem:continuous_on_support_representative_lemma} or \Cref{cor:continuous_on_support_representative_lemma_for_balls} to derive \ac{OM} functionals.

\begin{lemma}
	\label{lem:continuous_on_support_representative_lemma}
	Let $X$ be a vector space with a metric
	and $\mu \in \prob{X}$.
	Let $A\in\Borel{X}$ be a bounded neighbourhood of the origin.
	Let $\mu(F)=1$ for some $F\in \Borel{X}$, and $h \in Q(\mu)$.
	Assume that the shift density $r^\mu_h$ has a representative $\tilde{r}^\mu_h$ (i.e.\ $r^\mu_h-\tilde{r}^\mu_h=0$ in $L^1(\mu)$) such that $\tilde{r}^\mu_h |_{F}\colon F \to \Reals_{\geq 0}$ is continuous\footnote{This is a much weaker assumption than continuity of $\tilde{r}^\mu_h$ on $F$, which would mean that $\tilde{r}^\mu_h$ is continuous at each point of $F$ as a function on $X$.
		See also \citep[Lemma~4.6]{LieSullivan2018} for a result that only requires local continuity.
	}.
	Then, for all $x \in F\cap\supp(\mu)$, the limit below exists and
	\begin{equation}
		\label{eq:continuous_on_support_representative_lemma}
		\lim_{\varepsilon \searrow 0}\frac{\mu_h(\varepsilon A+x)}{\mu(\varepsilon A+x)}=\tilde{r}^\mu_h(x).
	\end{equation}
\end{lemma}

\begin{proof}
	Let $x \in X$ and $\varepsilon > 0$ be arbitrary.
	By definition of the shift density $r_{h}^{\mu}$,
	\[
	\mu_h(\varepsilon A+x)
	=
	\int_{\varepsilon A+x} r_{h}^{\mu}\, \rd \mu
	=
	\int_{\varepsilon A+x} \tilde{r}_{h}^{\mu}\, \rd \mu
	=
	\int_{(\varepsilon A+x)\cap F} \tilde{r}_{h}^{\mu}\, \rd \mu.
	\]
	By the hypotheses on $A$ and $x$, $\mu(\varepsilon A+x)>0$ for every $\varepsilon > 0$, and thus
	\[
	\inf_{y \in (\varepsilon A+x)\cap F} \tilde{r}_{h}^{\mu}(y)
	\leq
	\frac{\mu_h(\varepsilon A+x)}{\mu(\varepsilon A+x)}
	\leq
	\sup_{y \in (\varepsilon A+x)\cap F} \tilde{r}_{h}^{\mu}(y).
	\]
	Next, we will use the continuity of $\tilde{r}_{h}^{\mu} |_{F}$ on $F$ to show that as $\varepsilon\searrow 0$, the upper and lower bounds coincide.
	This will yield \eqref{eq:continuous_on_support_representative_lemma}.
	Let $x\in F$ and $\eta >0$.
	By continuity of $\tilde{r}_{h}^{\mu} |_{F}$, there exists $\delta > 0$ such that, for all $y \in \cBall{x}{\delta} \cap F$,
	\begin{equation}
		\label{eq:Continuity_of_r_in_proof}
		\absval{\tilde{r}_{h}^{\mu} (x) - \tilde{r}_{h}^{\mu} (y)}
		<
		\eta.
	\end{equation}
	Since $A$ is bounded, there exists $\varepsilon_{0} > 0$ such that, for all $0 < \varepsilon < \varepsilon_{0}$, $x + \varepsilon A \subseteq B_{\delta}(x)$.
	Hence, \eqref{eq:Continuity_of_r_in_proof} holds for all $0 < \varepsilon < \varepsilon_{0}$ and $y \in (\varepsilon A + x) \cap F$.
	Since $\eta > 0$ is arbitrary, this finishes the proof.
\end{proof}

\Cref{lem:continuous_on_support_representative_lemma} generalises \citep[Lemma~2.3]{AgapiouBurgerDashtiHelin2018} in two ways: it does not require symmetry or convexity of $A$, and it requires the continuity of the restriction of $\tilde{r}_{h}^{\mu}$ to some set of full measure $F$, instead of continuity of $\tilde{r}_{h}^{\mu}$ on the whole space $X$.
Continuity on $X$ was also assumed by \citet[Lemma~2]{HelinBurger2015}.
On the other hand, \citet[Lemma~2.3]{AgapiouBurgerDashtiHelin2018} do not assume $A$ to be a bounded neighbourhood of the origin.
However, the \change{fraction of small ball probabilities} on the left-hand side of \eqref{eq:continuous_on_support_representative_lemma} may be ill defined even if $\supp (\mu) = X$ and $A$ is symmetric and convex.
For example, if $\mu$ is an absolutely continuous measure on $(\Reals^{2},\Borel{\Reals^{2}})$ and $A = \{ 0 \} \times [-1,1]$ is a line segment, then $\mu(\varepsilon A+x) = 0$ for every $x$ and $\varepsilon$.
If $A$ is a bounded neighbourhood of the origin, then the expression on the left-hand side of \eqref{eq:continuous_on_support_representative_lemma} is well defined if and only if $x \in \supp (\mu)$.
In this case, we obtain the following result.

\begin{corollary}
	\label{cor:continuous_on_support_representative_lemma_for_balls}
	Let $X$ be a vector space with a metric, $\mu \in \prob{X}$ and $F\in\mathcal{B}(X)$ be a set of full measure.
	Assume that, for some $h \in Q(\mu)$, the shift density $r^\mu_h$ has a representative $\tilde{r}^\mu_h$ such that $\tilde{r}^\mu_h |_{F}\colon F \to \Reals_{\geq 0}$ is continuous.
	Then, for all $x_{\ast} \in F \cap \supp (\mu)$,
	\begin{equation}
		\label{eq:continuous_on_support_representative_formula_for_balls}
		\lim_{\varepsilon \searrow 0}\frac{\mu(\cBall{x_{\ast}+h}{\varepsilon})}{\mu(\cBall{x_{\ast}}{\varepsilon})}
		=
		\tilde{r}^\mu_{-h}(x_{\ast}).
	\end{equation}
	Assume that the above condition holds for any $h \in Q(\mu)$, let $x_{\ast} \in F\cap\supp (\mu)$ be arbitrary and $E(x_{\ast}) \defeq x_{\ast} + \Set{h \in Q(\mu)}{\tilde{r}^\mu_{-h}(x_{\ast}) \neq 0}$.
	Then
	\begin{equation}
		\label{eq:OM_functional_by_continuity_approach}
		I_{\mu,x_{\ast}} \colon E(x_{\ast}) \to \Reals,
		\qquad
		I_{\mu,x_{\ast}}(x)
		=
			- \log r_{x_{\ast}-x}^{\mu}(x_{\ast}), 
	\end{equation}
 	defines an \ac{OM} functional for $\mu$ on $E(x_{\ast})$.
\end{corollary}

\begin{proof}
	\change{Recall that \eqref{eq:shifted_measure} defines $
		\mu_{h}(A) \defeq \mu(A - h)$ for each $A \in \Borel{X}$. From this definition,} it follows that $\mu(A+h)=\mu_{-h}(A)$ and we obtain \eqref{eq:continuous_on_support_representative_formula_for_balls}.
    \change{Next, recall that \eqref{eq:r_mu_h_radon_nikodym_derivative_of_mu_h_wrt_mu} states that $
		\mu_{h}(A) = \int_{A} r^\mu_h(x) \, \mu(\rd x)$ for each $A \in \Borel{X}$. This implies that $r^\mu_0=1$ $\mu$-a.s. Hence $I_{\mu,x_{\ast}}(x_{\ast}) = 0$.}
	Now let $x \in E(x_{\ast})$, i.e.\ $x = x_{\ast} + h$ with $h\in Q(\mu)$ and $\tilde{r}^\mu_{-h}(x_{\ast}) \neq 0$.
	\change{Then \eqref{eq:OM_functional_by_continuity_approach} follows from}
	\[
	\lim_{\varepsilon \searrow 0}
	\frac{\mu(\cBall{x}{\varepsilon})}{\mu(\cBall{x_{\ast}}{\varepsilon})}
	=
	\lim_{\varepsilon \searrow 0}
	\frac{\mu(\cBall{x_{\ast}+h}{\varepsilon})}{\mu(\cBall{x_{\ast}}{\varepsilon})}
	=
	\tilde{r}^\mu_{-h}(x_{\ast})
	=
	\exp( \log \tilde{r}^\mu_{x_{\ast}-x}(x_{\ast}) )
	=
	\exp( I_{\mu,x_{\ast}}(x_{\ast}) - I_{\mu,x_{\ast}}(x) ).
	\]
\end{proof}

The derivation of $Q(\mu)$ and $r_{h}^{\mu}$ for product measures $\mu$ that satisfy \Cref{assump:basic_assumptions_for_product_measures} \ref{item:basic_assumption_product_mu_0}--\ref{item:basic_assumption_product_Shepp_condition} relies on a theorem of \citet{Kakutani1948} and a consequence of this theorem, due to \citet{Shepp1965}.
Therefore, we state both in \Cref{sec:equivalence}.
Below,
\begin{equation}
	\label{eq:Hellinger_integral}
	H(\mu,\nu) \defeq \int_\Omega \sqrt{\frac{\rd \mu}{\rd \lambda}\frac{\rd \nu}{\rd \lambda}} \, \rd \lambda
\end{equation}
denotes the \emph{Hellinger integral} of two probability measures $\mu$ and $\nu$ on the same measurable space $(\Omega,\mathcal{F})$, where $\lambda$ is another measure on $(\Omega,\mathcal{F})$ with $\mu,\nu\ll\lambda$.
Note that the value of $H(\mu, \nu)$ is independent of the choice of $\lambda$; see e.g.\ \citep[Chapter IV, \S{1.a}, Lemma~1.8]{JacodShiryaev2003}.

\begin{theorem}[Shift-quasi-invariance space and shift density $r_{h}^{\mu}$ of certain product measures]
	\label{thm:shift_quasi_invariance_space_for_product_measures}
	Let $\mu$ satisfy \Cref{assump:basic_assumptions_for_product_measures} \ref{item:basic_assumption_product_mu_0}--\ref{item:basic_assumption_product_Shepp_condition}.
	Then the shift-quasi-invariance space of $\mu$ is $Q(\mu) = \ell_{\gamma}^{2}$ and, for any $h \in Q(\mu)$ and $x \in \Reals^{\Naturals}$,
	\begin{equation}
		\label{eq:r_for_product_measures}
		r_{h}^{\mu}(x)
		=
		\prod_{k =1}^{\infty} \frac{\rho\bigl(\gamma_{k}^{-1} (x_{k} - m_{k} - h_{k}) \bigr)}{\rho\bigl(\gamma_{k}^{-1} (x_{k} - m_{k}) \bigr)}.
	\end{equation}
	Further, if \Cref{assump:basic_assumptions_for_product_measures} \ref{item:basic_assumption_product_full_measure} is satisfied, then the objects \change{$\mu_{\psi}, S_{\psi}$ and $T_{\psi}$ defined} in \Cref{notation:From_RN_to_Banach} satisfy $Q(\mu_{\psi}) = S_{\psi}(\ell_{\gamma}^{2})$, and, for any $h \in Q(\mu_{\psi})$ and $z\in Z$, $r_{h}^{\mu_{\psi}}(z) = r_{T_{\psi}(h)}^{\mu}(T_{\psi}(z))$.
\end{theorem}

\begin{proof}
	For $k \in\Naturals$, let $\nu_{k} \defeq \mu_{k}(\quark - h_{k})$, $\tilde{\mu}_{k} \defeq
	\mu_{0}$, $\tilde{\nu}_{k} \defeq \tilde{\mu}_{k}(\quark - \tilde{h}_{k})$, where $\tilde{h} = (\tilde{h}_{k})_{k \in \Naturals} \defeq (\gamma_{k}^{-1} h_{k})_{k \in \Naturals}$, and define
	\[
	\mu
	\defeq
	\bigotimes_{k\in\Naturals} \mu_{k},
	\qquad
	\nu
	\defeq
	\bigotimes_{k\in\Naturals} \nu_{k}
	=
	\mu(\quark - h),
	\qquad
	\tilde{\mu}
	\defeq
	\bigotimes_{k\in\Naturals} \tilde{\mu}_{k},
	\qquad
	\tilde{\nu}
	\defeq
	\bigotimes_{k\in\Naturals} \tilde{\nu}_{k}
	=
	\tilde{\mu}(\quark - \tilde{h}).
	\]
	From the definition of $\mu_k$ above, \ref{item:basic_assumption_product_mu_0} and \ref{item:basic_assumption_product_scaled_marginals}, we have $\ud\mu_k(x)=\gamma_{k}^{-1}\rho\big(\gamma_{k}^{-1}(x-m_k)\big)\ud x$.
	Using the definition of $\nu_k$ and the a.e.\ positivity of $\rho$ in \ref{item:basic_assumption_product_Shepp_condition}, it follows that $\mu_{k}\sim \nu_{k}$ and $\tilde{\mu}_{k}\sim \tilde{\nu}_{k}$ for all $k\in\Naturals$.
	Using the change of variables formula,
	\begin{align*}
		H(\mu_{k},\nu_{k})
		&=
		\gamma_{k}^{-1}
		\int_{\Reals} \rho\bigl(\gamma_{k}^{-1} (u - m_{k})\bigr)^{\frac{1}{2}}\,
		\rho\bigl(\gamma_{k}^{-1} (u - m_{k} - h_{k})\bigr)^{\frac{1}{2}}  \rd u
		\\
		&=
		\int_{\Reals} \rho(u)^{\frac{1}{2}}\,
		\rho(u-\tilde{h}_{k})^{\frac{1}{2}}  \rd u
		\\
		&=
		H(\tilde{\mu}_{k},\tilde{\nu}_{k}).
	\end{align*}
	Hence, by Kakutani's theorem (\Cref{thm:kakutanis_theorem}), $\mu\sim\nu$ if and only if $\tilde{\mu}\sim\tilde{\nu}$, and similarly $\mu\perp\nu$ if and only if $\tilde{\mu}\perp\tilde{\nu}$.
	Finally, Shepp's theorem (\Cref{thm:InvariantSpaceShepp}) implies the following:
	\begin{itemize}
		\item
		If $\sum_{k\in\Naturals} \tilde{h}_{k}^{2} = \sum_{k\in\Naturals} (h_{k}/\gamma_{k})^{2} < \infty$, then $\tilde{\mu} \sim \tilde{\nu}$.
		\item
		If $\sum_{k\in\Naturals} \tilde{h}_{k}^{2} = \sum_{k\in\Naturals} (h_{k}/\gamma_{k})^{2} = \infty$, then $\tilde{\mu} \perp \tilde{\nu}$.
	\end{itemize}
	This proves $Q(\mu) = \ell_{\gamma}^{2}$, where we used that $\ell_{\gamma}^{2} \subseteq X$ by \Cref{cor:handy_inclusion_of_weighted_lp_spaces_general}, while \eqref{eq:r_for_product_measures} follows directly from \Cref{thm:kakutanis_theorem}.
	For the final statement first note that, since $\mu(X) = 1$ by assumption, we have, for any $B\in\Borel{Z}$ and $h\in Z$,
	\[
	\mu_{\psi}(B)
	=
	\mu(S_{\psi}^{-1}(B))
	=
	\mu(T_{\psi}(B) \cap X)
	=
	\mu(T_{\psi}(B)),
	\qquad
	\mu_{\psi}(B-h)
	=
	\mu(T_{\psi}(B) - T_{\psi}(h)).
	\]
	Hence, for $h\in Z$, the shift density $r_{h}^{\mu_{\psi}}$ on $Z$ exists if and only if the shift density $r_{T_{\psi}(h)}^{\mu}$ on $\Reals^{\Naturals}$ exists, in which case $r_{h}^{\mu_{\psi}}(z) = r_{T_{\psi}(h)}^{\mu}(T_{\psi}(z))$.
\end{proof}

Having identified the shift-quasi-invariance space $Q(\mu)$ and the shift density $r_{h}^{\mu}$, the second step in the continuity approach involves finding a representative $\tilde{r}_{h}^{\mu}$ and a sufficiently large subset $F$ of $X$ such that the restriction of $\tilde{r}_{h}^{\mu}$ to $F$ is continuous.
The third step is then to apply either \Cref{lem:continuous_on_support_representative_lemma} or \Cref{cor:continuous_on_support_representative_lemma_for_balls}.
We do not pursue the continuity approach further because the second step is difficult to carry out and because a more direct approach yielded the desired results.
We describe the direct approach in the next section.

\subsection{Direct approach}
\label{ssec:direct_approach}

The following definition and theorem provide the basis for establishing the \ac{OM} functional for the product measures defined in \Cref{assump:basic_assumptions_for_product_measures}.
We demonstrate this by applying both to the Cauchy measure in \Cref{cor:OM_Cauchy}, and to the Besov-$p$ measure with $1 \leq p \leq 2$ in \Cref{thm:OM_functional_for_Besov_with_p_in_1_2}.

\change{Recall that \ref{item:basic_assumption_product_mu_0} assumes that the reference measure $\mu_{0}$ on $\Reals$ has a continuous, symmetric density $\rho$ decreasing on $\Reals_{\geq 0}$, and \ref{item:basic_assumption_product_scaled_marginals} assumes that each measure $\mu_{k}$ on $\Reals$ is obtained from $\mu_{0}$ by an affine change of variables.}
\begin{definition}
	\label{def:logdensity}
	Under \Cref{assump:basic_assumptions_for_product_measures} \ref{item:basic_assumption_product_mu_0}--\ref{item:basic_assumption_product_scaled_marginals} we define the \defterm{negative log-density} $\logden \colon \Reals \to \eReals_{\geq 0}$ by
	\begin{equation}
		\label{eq:negative_log-density}
		\logden(u)
		\defeq
		- \log \frac{\rho(u)}{\rho(0)}
		=
		\log \rho(0) - \log \rho(u),
	\end{equation}
	and the \defterm{formal negative log-density} $\logden_{\gamma,m} \colon X \to \eReals_{\geq 0}$ by
	\begin{equation}
		\label{eq:formal_negative_log-density}
		\logden_{\gamma,m}(h)
		\defeq
		\sum_{k \in \Naturals}\logden(\gamma_{k}^{-1}(h_{k} - m_{k})).
	\end{equation}
	Further, we set $E_{\gamma,m} \defeq \Set{h \in X}{\logden_{\gamma,m}(h) < \infty}$.
	Similarly, using \Cref{notation:From_RN_to_Banach}, we define the \defterm{formal negative log-density} $\logden_{\gamma,m} \colon Z \to \eReals_{\geq 0}$ by
	\[
	\logden_{\gamma,m,\psi}(h)
	\defeq
	\begin{cases}
	\logden_{\gamma,m}(T_{\psi}(h)) & \text{if } T_{\psi}(h) \in X,
	\\
	+\infty & \text{otherwise,}
	\end{cases}
	\]
	and
	$E_{\gamma,m,\psi} \defeq \Set{h \in Z}{\logden_{\gamma,m,\psi}(h) < \infty} = S_{\psi}(E_{\gamma,m})$.
\end{definition}

Note that, by \Cref{assump:basic_assumptions_for_product_measures} \ref{item:basic_assumption_product_mu_0}, $\logden|_{\Reals_{\geq 0}} \colon \Reals_{\geq 0} \to \Reals_{\geq 0}$ is a strictly monotonically increasing bijection.

\change{Recall that \Cref{assump:basic_assumptions_for_product_measures} \ref{item:basic_assumption_product_full_measure} refers to the assumption that $X=\ell_{\alpha}^{p}$ and $\mu(X)=1$,  \ref{item:basic_assumption_product_integrable_second_derivative} assumes that the reference measure $\mu_{0}$ has density $\rho \in C^{2}(\Reals)$ such that $\rho''\in L^1(\Reals)$, and \ref{item:basic_assumption_product_Besov_p_1_2} assumes that $\mu=B^{s}_{p}$ is a Besov measure with $1\leq p\leq 2$ and $\alpha=\delta$.}
\begin{theorem}
	\label{thm:OM_for_product_measures}
	Under \Cref{assump:basic_assumptions_for_product_measures} \ref{item:basic_assumption_product_full_measure}--\ref{item:basic_assumption_product_scaled_marginals},
	\begin{align}
	\label{eq:Limit_ball_ratio_product_measures_upper_bound}
	\lim_{r \searrow 0} \frac{\mu( \cBall{h}{r} )}{\mu( \cBall{m}{r})}
	&\leq
	\begin{cases}
	\exp\left( - \logden_{\gamma,m}(h)  \right) & \text{if } h \in E_{\gamma,m},
	\\
	0 & \text{if } h \notin E_{\gamma,m}.
	\end{cases}
	\intertext{
	In particular, property $M(\mu, E_{\gamma,m})$ is satisfied and, if $I_{\mu} \colon X \to \eReals$ is an (extended) \ac{OM} functional for $\mu$ with $I_{\mu}(m) = 0$, then $I_{\mu} \geq \logden_{\gamma,m}$.
	If, in addition, either \Cref{assump:basic_assumptions_for_product_measures} \ref{item:basic_assumption_product_integrable_second_derivative} or \ref{item:basic_assumption_product_Besov_p_1_2} is satisfied, then
	}
	\label{eq:Limit_ball_ratio_product_measures_lower_bound_l_2}
	\lim_{r \searrow 0} \frac{\mu( \cBall{h}{r} )}{\mu( \cBall{m}{r})}
	&\geq
	\exp\left( - \logden_{\gamma,m}(h)  \right)
	\quad
	\text{if } h \in E_{\gamma,m} \cap (m + \ell_{\gamma}^{2}).
	\end{align}
	In particular, in this case and under the additional assumption that \change{$E_{\gamma,m} \subseteq m + \ell_{\gamma}^{2}$}, $I_{\mu} = \logden_{\gamma,m} \colon X \to \eReals$ is an (extended) \ac{OM} functional for $\mu$.
\end{theorem}

\begin{proof}
The technical proof is given in \Cref{section:Proof_Theorem_OM_for_product_measures}.
\end{proof}

Similar statements follow for the Banach space $Z$ in \Cref{notation:From_RN_to_Banach} under the assumption that $S_{\psi}$ is an isometry.

\begin{corollary}
\label{cor:OM_for_product_measures_Banach}
Using \Cref{notation:From_RN_to_Banach}, assuming $S_{\psi}$ to be an isometry, and assuming that \Cref{assump:basic_assumptions_for_product_measures} \ref{item:basic_assumption_product_full_measure}--\ref{item:basic_assumption_product_scaled_marginals} hold,
\begin{align}
	\label{eq:Limit_ball_ratio_product_measures_upper_bound_Banach}
	\lim_{r \searrow 0} \frac{\mu_{\psi}( \cBall{h}{r} )}{\mu_{\psi}( \cBall{S_{\psi} m}{r})}
	&\leq
	\begin{cases}
		\exp\left( - \logden_{\gamma,m,\psi}(h)  \right) & \text{if } h \in E_{\gamma,m,\psi},
		\\
		0 & \text{if } h \notin E_{\gamma,m,\psi}.
	\end{cases}
	\intertext{
		In particular, property $M(\mu_{\psi}, E_{\gamma,m,\psi})$ is satisfied and, if  $I_{\mu_{\psi}} \colon Z \to \eReals$ is an \ac{OM} functional for $\mu_{\psi}$ with $I_{\mu_{\psi}}(T_{\psi}(m)) = 0$, then $I_{\mu_{\psi}} \geq \logden_{\gamma,m,\psi}$.
		If, in addition, either \Cref{assump:basic_assumptions_for_product_measures} \ref{item:basic_assumption_product_integrable_second_derivative} or \ref{item:basic_assumption_product_Besov_p_1_2} is satisfied, then
	}
	\label{eq:Limit_ball_ratio_product_measures_lower_bound_l_2_Banach}
	\lim_{r \searrow 0} \frac{\mu_{\psi}( \cBall{h}{r} )}{\mu_{\psi}( \cBall{S_{\psi}m}{r})}
	&\geq
	\exp\left( - \logden_{\gamma,m,\psi}(h)  \right)
	\quad
	\text{if } h \in E_{\gamma,m,\psi} \cap S_{\psi}(m + \ell_{\gamma}^{2}).
\end{align}
In particular, in this case and under the additional assumption that \change{$E_{\gamma,m} \subseteq m + \ell_{\gamma}^{2}$}, $I_{\mu_{\psi}} = \logden_{\gamma,m,\psi} \colon Z \to \eReals$ is an (extended) \ac{OM} functional for $\mu_{\psi}$.
\end{corollary}

\begin{proof}
By \Cref{lemma:From_RN_to_Banach}, \eqref{eq:Limit_ball_ratio_product_measures_upper_bound_Banach} and \eqref{eq:Limit_ball_ratio_product_measures_lower_bound_l_2_Banach} follow directly from \eqref{eq:Limit_ball_ratio_product_measures_upper_bound} and \eqref{eq:Limit_ball_ratio_product_measures_lower_bound_l_2}.
\end{proof}

\Cref{thm:OM_for_product_measures} yields the full OM functional for a limited class of product measures.
We conjecture that the conclusions of \Cref{thm:OM_for_product_measures} hold for a larger class of product measures.

\begin{conjecture}[\ac{OM} functional of product measures]
	\label{conjecture:OM_for_product_measures}
	Under \Cref{assump:basic_assumptions_for_product_measures} \ref{item:basic_assumption_product_full_measure}--\ref{item:basic_assumption_product_scaled_marginals},
	\begin{align}
		\label{eq:Limit_ball_ratio_product_measures_conjecture}
		\lim_{r \searrow 0} \frac{\mu( \cBall{h}{r} )}{\mu( \cBall{m}{r})}
		&=
		\begin{cases}
			\exp\left( - \logden_{\gamma,m}(h)  \right) & \text{if } h \in E_{\gamma,m},
			\\
			0 & \text{if } h \notin E_{\gamma,m}.
		\end{cases}
	\end{align}
In particular, property $M(\mu, E_{\gamma,m})$ is satisfied and $I_{\mu} \colon X \to \eReals$ with $I_{\mu} = \logden_{\gamma,m}$ defines an \ac{OM} functional for $\mu$.
\end{conjecture}

\change{The following two theorems refer to \eqref{equ:OM_from_RN_to_Banach}, which we recall below:
\begin{equation*}
		I_{\mu_{\psi}} \colon Z \to \eReals,
		\qquad
		I_{\mu_{\psi}} (h)
		=
		\begin{cases}
			I_{\mu}(T_{\psi}h) & \text{if } h \in \range S_{\psi},
			\\
			+\infty & \text{otherwise.}
		\end{cases}
	\end{equation*}
}

The following result concerns equicoercivity of a sequence of \ac{OM} functionals.
It assumes that one is given a sequence of probability measures, where each probability measure $\mu^{(n)} \in \prob{\Reals^{\Naturals}}$ is defined by, in the sense of \Cref{assump:basic_assumptions_for_product_measures} \ref{item:basic_assumption_product_full_measure}--\ref{item:basic_assumption_product_scaled_marginals}, an absolutely continuous reference measure $\mu_{0}^{(n)} \in \prob{\Reals}$, a shift vector $m^{(n)} \in X = \ell_{\alpha}^{p}$, and a scaling vector $\gamma^{(n)} \in \Reals_{>0}^{\Naturals}$, $n\in \Naturals\cup \{ \infty \}$ \change{(note that $\gamma^{(n)} \in X$ by \Cref{lemma:Connection_between_scaling_sequences})}.
Furthermore, it assumes that each probability measure has an \ac{OM} functional.
The result states that if the sequence of probability measures $\mu^{(n)}$ converges to $\mu^{(\infty)}$ in the sense that both the sequence of shift vectors and the sequence of scaling vectors converge in $X$ to the corresponding pair of shift and scaling vectors, and if the sequence of Lebesgue densities of the reference measures converges pointwise, then the sequence of \ac{OM} functionals is equicoercive.
\begin{theorem}[Equicoercivity for product measures]
\label{thm:Equicoercivity_for_product_measures}
For $n\in\Naturals \cup \{ \infty \}$, let $\mu^{(n)}\in \prob{X}$ be probability measures on the same space $X = \ell_{\alpha}^{p}$ that satisfy \Cref{assump:basic_assumptions_for_product_measures} \ref{item:basic_assumption_product_full_measure}--\ref{item:basic_assumption_product_scaled_marginals}
with shift parameters $m^{(n)} \in X$, scale parameters $\gamma^{(n)} \in \Reals_{> 0}^{\Naturals}$ and probability densities $\rho^{(n)}$ of the measures $\mu_{0}^{(n)}\in \prob{\Reals}$.
If \ac{OM} functionals $I_{\mu^{(n)}} \colon X \to \eReals$ with $I_{\mu^{(n)}}(m^{(n)}) = 0$ exist for all $n\in\Naturals \cup \{ \infty \}$ and if  $\norm{m^{(n)} - m^{(\infty)}}_{X}\to 0$, $\norm{\gamma^{(n)} - \gamma^{(\infty)}}_{X}\to 0$ and $\rho^{(n)}\to\rho^{(\infty)}$ (pointwise) as $n\to\infty$, then the sequence $(I_{\mu^{(n)}})_{n\in\Naturals}$ is equicoercive.
Further, using \Cref{notation:From_RN_to_Banach} and assuming $S_{\psi}$ to be an isometry, the sequence $(I_{\mu_{\psi}^{(n)}})_{n\in\Naturals}$ defined by \eqref{equ:OM_from_RN_to_Banach} is equicoercive.
\end{theorem}

\begin{proof}
	By \Cref{assump:basic_assumptions_for_product_measures} \ref{item:basic_assumption_product_mu_0}, the negative log-densities $\logden^{(n)} \colon \Reals \to \eReals_{\geq 0}$ are symmetric and their restrictions $\logden^{(n)}|_{\Reals_{\geq 0}} \colon \eReals_{\geq 0} \to \eReals_{\geq 0}$ are strictly monotonically increasing bijections.
	Let $t\geq 0$ and $a_{n} \defeq (\logden^{(n)}|_{\Reals_{\geq 0}})^{-1}(t)$.
	Since $\rho^{(n)}$ converges pointwise to $\rho^{(\infty)}$ by assumption, $a_{n} \to a_{\infty}$ as $n\to\infty$.
	Further, by \Cref{thm:OM_for_product_measures},
	\begin{equation}
		\label{eq:OM_inequality_product_measures}
		I_{\mu^{(n)}}(x)
		\geq
		\logden_{\gamma^{(n)},m^{(n)}}^{(n)}(x)
		=
		\sum_{k \in \Naturals}\logden^{(n)}\bigl( (\gamma_{k}^{(n)})^{-1}(x_{k} - m_{k}^{(n)})\bigr),
		\qquad
		x \in X.
	\end{equation}

	The proof is structured around the following four steps, of which the second and fourth are straightforward.

	\noindent\textbf{Step 1.}
	The operators
	\[
		T^{(n)} \colon \ell^{\infty} \to \ell_{\alpha}^{p},
		\qquad
		(v_{k})_{k\in\Naturals}
		\mapsto
		(\gamma_{k}^{(n)} v_{k})_{k\in\Naturals},
		\qquad
		n\in\Naturals \cup \{ \infty \},
	\]
	are well defined, compact and $\norm{a_{n} T^{(n)} - a_{\infty} T^{(\infty)}} \to 0$ as $n\to\infty$.

	\smallskip

	\noindent\textbf{Step 2.}
	It follows that the sets
	\[
		K_{t}^{(n)}
		\defeq
		m^{(n)} + a_{n}\, T^{(n)} \cBallExtra{0}{1}{\ell^{\infty}}
		=
		\prod_{k \in \Naturals} [ m_{k}^{(n)} - \gamma_{k}^{(n)} a_{n} , m_{k}^{(n)} + \gamma_{k}^{(n)} a_{n} ]
	\]
	are pre-compact and, by \eqref{eq:OM_inequality_product_measures},
	\begin{align*}
		I_{\mu^{(n)}}^{-1}([-\infty,t])
		&\subseteq
		\Set{x \in X}{\logden^{(n)}\bigl( (\gamma_{k}^{(n)})^{-1}(x_{k} - m_{k}^{(n)})\bigr) \leq t \text{ for each } k\in\Naturals}
		\\
		&=
		\Set{x \in X}{\absval{x_{k} - m_{k}^{(n)}} \leq \gamma_{k}^{(n)} a_{n} \text{ for each } k\in\Naturals}
		\\
		&=
		K_{t}^{(n)}.
	\end{align*}

	\smallskip

	\noindent\textbf{Step 3.}
	$K_{t}^{\circ} \defeq \bigcup_{n\in \Naturals} K_{t}^{(n)}$ is sequentially pre-compact.
	Hence $K_{t} \defeq \overline{K_{t}^{\circ}}$ is compact, which proves equicoercivity of $(I_{\mu^{(n)}})_{n\in\Naturals}$.
	Note that for $t < 0$ there is nothing to prove, since $I_{\mu^{(n)}}^{-1}([-\infty,t]) = \emptyset$ for each $n \in \Naturals \cup \{ \infty \}$ in this case.

	\smallskip

	\noindent\textbf{Step 4.}
	Equicoercivity of $(I_{\mu_{\psi}^{(n)}})_{n\in\Naturals}$ follows directly from \Cref{lemma:From_RN_to_Banach}. \change{Recall that this lemma transforms an \ac{OM} functional on the sequence space $X$ into an \ac{OM} functional on the separable Banach space $Z$, where $X$ and $Z$ are related by the synthesis operator $S_{\psi}:X\to Z$ and coordinate operator $T_{\Psi}:Z\to \Reals^{\Naturals}$.}

	\medskip
	
	We now give the proofs of the non-trivial first and third steps.
	
	\medskip

	\noindent\textbf{Proof of Step 1.}
	Let $n \in \Naturals \cup \{ \infty \}$.
	Since $\gamma^{(n)} \in \ell_{\alpha}^{p}$ by \Cref{lemma:Connection_between_scaling_sequences}, H\"{o}lder's inequality implies, for any $v \in \ell^{\infty}$,
	\[
		\norm{T^{(n)}v}_{\ell_{\alpha}^{p}}^{p}
		=
		\sum_{k \in \Naturals} \absval{\alpha_{k}^{-1} \gamma_{k}^{(n)} v_{k}}^{p}
		\leq
		\norm{\gamma^{(n)}}_{\ell_{\alpha}^{p}}^{p} \norm{v}_{\ell^{\infty}}^{p}
		<
		\infty,
	\]
	proving well-definedness of $T^{(n)}$.
	Consider the finite-rank operators
	\[
		T_{m}^{(n)} \colon \ell^{\infty} \to \ell_{\alpha}^{p},
		\qquad
		(v_{k})_{k\in\Naturals}
		\mapsto
		(\gamma_{k}^{(n)} v_{k})_{k = 1,\dots,m},
		\qquad
		m\in\Naturals.
	\]
	Then $\norm{T_{m}^{(n)} - T^{(n)}} \to 0$ as $m \to \infty$, since H\"{o}lder's inequality implies, for any $v \in \ell^{\infty}$ with $\norm{v}_{\ell^{\infty}} \leq 1$,
	\[
		\norm{(T_{m}^{(n)} - T^{(n)})\, v}_{\ell_{\alpha}^{p}}^{p}
		=
		\sum_{k > m} \absval{\alpha_{k}^{-1} \gamma_{k}^{(n)} v_{k}}^{p}
		\leq
		\norm{v}_{\ell^{\infty}}^{p}
		\sum_{k > m} \absval{\alpha_{k}^{-1} \gamma_{k}^{(n)}}^{p}
		\leq
		\sum_{k > m} \absval{\alpha_{k}^{-1} \gamma_{k}^{(n)}}^{p},
	\]
	where the last term is independent of $v$ and goes to $0$ as $m\to\infty$ since $\gamma^{(n)} \in \ell_{\alpha}^{p}$.
	Hence, $T^{(n)}$ is a compact operator.
	Finally, $\norm{T^{(n)} - T^{(\infty)}} \to 0$ as $n\to\infty$, since H\"{o}lder's inequality implies, for any $v \in \ell^{\infty}$ with $\norm{v}_{\ell^{\infty}} \leq 1$,
	\[
		\norm{(T^{(n)} - T^{(\infty)})\, v}_{\ell_{\alpha}^{p}}^{p}
		=
		\sum_{k \in \Naturals} \Absval{\alpha_{k}^{-1} (\gamma_{k}^{(n)} - \gamma_{k}^{(\infty)}) v_{k}}^{p}
		\leq
		\norm{\gamma^{(n)}-\gamma^{(\infty)}}_{\ell_{\alpha}^{p}}^{p}
		\norm{v}_{\ell^{\infty}}^{p}
		\leq
		\norm{\gamma^{(n)}-\gamma^{(\infty)}}_{\ell_{\alpha}^{p}}^{p},
	\]
	where the last term is independent of $v$ and goes to $0$ as $n\to\infty$ by assumption.
	It follows that
	\begin{align*}
		\norm{a_{n} T^{(n)} - a_{\infty} T^{(\infty)}}
		&\leq
		\norm{a_{n} (T^{(n)} - T^{(\infty)})}
		+
		\norm{(a_{n} - a_{\infty}) T^{(\infty)}}
		\\
		&\leq
		(\sup_{n \in \Naturals} \absval{a_{n}})
		\underbrace{\norm{(T^{(n)} - T^{(\infty)})}}_{\to\, 0}
		+
		\underbrace{\absval{a_{n} - a_{\infty}}}_{\to\, 0}
		\norm{T^{(\infty)}}
		\\
		&\xrightarrow[n\to\infty]{}
		0.
	\end{align*}

	\smallskip

	\noindent\textbf{Proof of Step 3.}
	Let $(x^{(\nu)})_{\nu \in \Naturals}$ be a sequence in $K_{t}^{\circ}$.
	If there exists $n \in \Naturals$ such that $x^{(\nu)} \in K_{t}^{(n)}$ infinitely often, then there is nothing to show, since $K_{t}^{(n)}$ is pre-compact.
	Otherwise, there exist subsequences $(x^{(\nu_{j})})_{j \in \Naturals}$ and $(K_{t}^{(n_{j})})_{j \in \Naturals}$ such that $x^{(\nu_{j})} \in K_{t}^{(n_{j})}$ for each $j \in \Naturals$.
	By the definition of $K_{t}^{(n)}$, there exist $v^{(j)} \in \cBallExtra{0}{1}{\ell^{\infty}}$ such that $a_{n_{j}} T^{(n_{j})} v^{(j)} = x^{(\nu_{j})} - m^{(n_{j})}$.
	Since $K_{t}^{(\infty)}$ is pre-compact, the sequence $(w^{(j)})_{j\in\Naturals}$ given by $w^{(j)} \defeq m^{(\infty)} + a_{\infty} T^{(\infty)} v^{(j)} \in K_{t}^{(\infty)}$ has a subsequence --- which for simplicity we also denote by $(w^{(j)})_{j \in \Naturals}$ --- that converges to some element $w \in X$.
	It follows that, as $j\to\infty$,
	\begin{align*}
		\norm{x^{(\nu_{j})} - w}_{X}
		& \leq
		\norm{(x^{(\nu_{j})} - m^{(n_{j})}) - (w^{(j)} - m^{(\infty)})}_{X}
		+ \norm{m^{(n_{j})} - m^{(\infty)}}_{X}
		+ \norm{w^{(j)} - w}_{X}
		\\
		&\leq
		\underbrace{\norm{a_{n_{j}}T^{(n_{j})} - a_{\infty} T^{(\infty)}}}_{\to 0}
		\underbrace{\norm{v^{(j)}}_{\ell^{\infty}}}_{\leq 1}
		+ \underbrace{\norm{m^{(n_{j})} - m^{(\infty)}}_{X}}_{\to 0}
		+ \underbrace{\norm{w^{(j)} - w}_{X}}_{\to 0}
		\\
		&\to
		0,
	\end{align*}
	and thus $(x^{(\nu)})_{\nu \in \Naturals}$ has a convergent subsequence and $K_{t}^{\circ}$ is sequentially pre-compact.
\end{proof}

The following result concerns $\Gamma$-convergence of \ac{OM} functionals.
As in \Cref{thm:Equicoercivity_for_product_measures}, one is given a sequence of probability measures, where each probability measure $\mu^{(n)}$ is defined by an absolutely continuous reference measure $\mu_{0}^{(n)}$, a shift vector $m^{(n)}$, and a scaling vector $\gamma^{(n)}$, and each probability measure has an \ac{OM} functional.
Again, we assume convergence in $X$ of the sequence of shift vectors and the sequence of scaling vectors.
However, we replace the assumption of pointwise convergence of the sequence of Lebesgue densities in \Cref{thm:Equicoercivity_for_product_measures} with the assumption of local uniform convergence from below of the negative log-densities and assume the \ac{OM} functionals to have the specific form $I_{\mu^{(n)}} = \logden_{\gamma^{(n)},m^{(n)}}^{(n)}$.
Under these assumptions, we obtain $\Gamma$-convergence of the \ac{OM} functionals.

\begin{theorem}[$\mathsf{\Gamma}$-convergence for product measures]
	\label{thm:Gamma_convergence_for_product_measures}
	For $n\in\Naturals \cup \{ \infty \}$, let $\mu^{(n)}\in \prob{X}$ be probability measures on the same space $X = \ell_{\alpha}^{p}$ that satisfy \Cref{assump:basic_assumptions_for_product_measures} \ref{item:basic_assumption_product_full_measure}--\ref{item:basic_assumption_product_scaled_marginals}
	with shift parameters $m^{(n)} \in X$, scale parameters $\gamma^{(n)} \in \Reals_{> 0}^{\Naturals}$ and probability densities $\rho^{(n)}$ of the measures $\mu_{0}^{(n)}\in \prob{\Reals}$.
	Let $\logden^{(n)} \colon \Reals \to \eReals_{\geq 0}$ and $\logden_{\gamma^{(n)},m^{(n)}}^{(n)} \colon X \to \eReals$ denote the corresponding (formal) negative log-densities (see \Cref{def:logdensity}).
	Assume that $\norm{m^{(n)} - m^{(\infty)}}_{X}\to 0$, $\norm{\gamma^{(n)} - \gamma^{(\infty)}}_{X}\to 0$, that $\logden^{(n)} \to \logden^{(\infty)}$ locally uniformly as $n\to\infty$, that $\logden^{(n)} \leq \logden^{(\infty)}$ for all but finitely many $n\in\Naturals$ and that $I_{\mu^{(n)}} \colon X \to \eReals$ with $I_{\mu^{(n)}} = \logden_{\gamma^{(n)},m^{(n)}}^{(n)}$ defines an \ac{OM} functional for $\mu^{(n)}$ for each $n\in\Naturals \cup \{ \infty \}$.
	Then $I_{\mu^{(n)}} \xrightarrow[n \to \infty]{\Gamma} I_{\mu^{(\infty)}}$.
	Further, using \Cref{notation:From_RN_to_Banach} and assuming $S_{\psi}$ to be an isometry, $I_{\mu_{\psi}^{(n)}} \xrightarrow[n \to \infty]{\Gamma} I_{\mu_{\psi}^{(\infty)}}$ where $I_{\mu_{\psi}^{(n)}},\ n\in\Naturals$, are defined by \eqref{equ:OM_from_RN_to_Banach}.
\end{theorem}

\begin{proof}
	For the $\Gamma$-$\liminf$ inequality, let $(x^{(n)})_{n\in\Naturals}$ be a sequence in $X$ that converges to $x\in X$ as $n\to\infty$.
	Then, by Fatou's lemma,
	\begin{align*}
		I_{\mu^{(\infty)}}(x)
		&=
		\sum_{k \in \Naturals} \logden^{(\infty)}\bigl( (\gamma_{k}^{(\infty)})^{-1}(x_{k} - m_{k}^{(\infty)}) \bigr) & & \text{by assumption}
		\\
		&=
		\sum_{k \in \Naturals} \lim_{n \to \infty} \logden^{(n)}\bigl( (\gamma_{k}^{(n)})^{-1}(x_{k}^{(n)} - m_{k}^{(n)})\bigr) & & \text{since $\logden^{(n)} \to \logden^{(\infty)}$ locally uniformly}
		\\
		&\leq
		\liminf_{n \to \infty} \sum_{k \in \Naturals} \logden^{(n)}\bigl( (\gamma_{k}^{(n)})^{-1}(x_{k}^{(n)} - m_{k}^{(n)})\bigr) & & \text{by Fatou's lemma}
		\\
		&=
		\liminf_{n \to \infty} I_{\mu^{(n)}}(x^{(n)}) & & \text{by assumption}.
	\end{align*}
	Note that Fatou's lemma is general enough to handle extended real-valued sequences, so we do not need to treat cases such as $I_{\mu^{(\infty)}}(x) = \infty$ separately.
	For the $\Gamma$-$\limsup$ inequality, let $x\in X$ and choose the sequence $(x^{(n)})_{n\in\Naturals}$ in $X$ by
	\begin{equation}
		\label{eq:Choice_sequence_Gamma_lim_sup_inequality_product_measures}
		x_{k}^{(n)}
		\defeq
		m_{k}^{(n)} + \frac{\gamma_{k}^{(n)}}{\gamma_{k}^{(\infty)}} (x_{k} - m_{k}^{(\infty)}).
	\end{equation}
	If $I_{\mu^{(\infty)}}(x) = \logden_{\gamma^{(\infty)},m^{(\infty)}}^{(\infty)}(x) = \infty$, then there is nothing to show (simply choose $x^{(n)} \defeq x$ for all $n \in \Naturals$).
	Now suppose that $I_{\mu^{(\infty)}}(x) = \logden_{\gamma^{(\infty)},m^{(\infty)}}^{(\infty)}(x)$ is finite.
	\change{By \ref{item:basic_assumption_product_mu_0} -- the assumption that the reference density $\rho$ is continuous, symmetric, and monotonically decreasing -- and the formula \eqref{eq:negative_log-density} -- which states that $		\logden(u)	\defeq		- \log \frac{\rho(u)}{\rho(0)}		=
		\log \rho(0) - \log \rho(u)$ -- it follows that }$\logden^{(\infty)}$ is monotonically increasing, with $\logden^{(\infty)}(x) \to \infty$ as $\absval{x} \to \infty$.
	If the terms $\frac{x_{k} - m_{k}^{(\infty)}}{\gamma_{k}^{(\infty)}}$ are unbounded, then this implies that the $\logden^{(\infty)}(\frac{x_{k} - m_{k}^{(\infty)}}{\gamma_{k}^{(\infty)}})$ are unbounded, and hence that $I_{\mu^{(\infty)}}(x)$ is not finite.
	By taking the contrapositive, it follows that if $I_{\mu^{(\infty)}}(x) = \logden_{\gamma^{(\infty)},m^{(\infty)}}^{(\infty)}(x)$ is finite, then

	\[
	S \defeq \sup_{k\in\Naturals} \Absval{\frac{x_{k} - m_{k}^{(\infty)}}{\gamma_{k}^{(\infty)}}} < \infty.
	\]
	By \Cref{lemma:Connection_between_scaling_sequences}, $\gamma^{(n)} \in \ell_{\alpha}^{p}$.
	By \eqref{eq:Choice_sequence_Gamma_lim_sup_inequality_product_measures}, $(x_{k}^{(n)} - m_{k}^{(n)}) - (x_{k} - m_{k}^{(\infty)})=(x_{k}-m_{k}^{(\infty)})(1-\tfrac{\gamma_{k}^{(n)}}{\gamma_{k}^{(\infty)}})$.
	Thus,
	\begin{align*}
		\norm{(x^{(n)} - m^{(n)}) - (x - m^{(\infty)})}_{X}^{p}
		&=
		\sum_{k \in \Naturals} \Absval{\frac{(x_{k}^{(n)} - m_{k}^{(n)}) - (x_{k} - m_{k}^{(\infty)})}{\alpha_{k}}}^{p}
		\\
		&=
		\sum_{k \in \Naturals} \Absval{\frac{x_{k} - m_{k}^{(\infty)}}{\gamma_{k}^{(\infty)}}}^{p}
		\Absval{\frac{\gamma_{k}^{(n)} - \gamma_{k}^{(\infty)}}{\alpha_{k}}}^{p}
		\\
		&\leq
		S^{p} \, \norm{\gamma^{(n)} - \gamma^{(\infty)}}_{\ell_{\alpha}^{p}}^{p}
		\\
		&\xrightarrow[n\to\infty]{}
		0.
	\end{align*}
	It follows that
	\[
	\norm{x^{(n)} - x}_{X}^{p}
	\leq
	\norm{(x^{(n)} - m^{(n)}) - (x - m^{(\infty)})}_{X}^{p}
	+
	\norm{m^{(n)} - m^{(\infty)}}_{X}^{p}
	\xrightarrow[n\to\infty]{}
	0.
	\]
	Using the reverse Fatou lemma and that $\logden^{(n)} \leq \logden^{(\infty)}$ for all but finitely many $n\in\Naturals$,
	\begin{align*}
		I_{\mu^{(\infty)}}(x)
		&=
		\sum_{k \in \Naturals} \logden^{(\infty)}\bigl( (\gamma_{k}^{(\infty)})^{-1}(x_{k} - m_{k}^{(\infty)}) \bigr)
		& & \text{by assumption}
		\\
		&=
		\sum_{k \in \Naturals} \lim_{n \to \infty} \logden^{(n)}\bigl( (\gamma_{k}^{(\infty)})^{-1}(x_{k} - m_{k}^{(\infty)}) \bigr)
		& & \text{since $\logden^{(n)} \to \logden^{(\infty)}$ pointwise}
		\\
		&\geq
		\limsup_{n \to \infty} \sum_{k \in \Naturals} \logden^{(n)}\bigl( (\gamma_{k}^{(\infty)})^{-1}(x_{k} - m_{k}^{(\infty)}) \bigr)
		& & \text{by the reverse Fatou lemma}
		\\
		&=
		\limsup_{n \to \infty} \sum_{k \in \Naturals} \logden^{(n)}\bigl( (\gamma_{k}^{(n)})^{-1}(x_{k}^{(n)} - m_{k}^{(n)})\bigr)
		& & \text{by \eqref{eq:Choice_sequence_Gamma_lim_sup_inequality_product_measures}}
		\\
		&=
		\limsup_{n \to \infty} I_{\mu^{(n)}}(x^{(n)})
		& & \text{by assumption}.
	\end{align*}
	$I_{\mu_{\psi}^{(n)}} \xrightarrow[n \to \infty]{\Gamma} I_{\mu_{\psi}^{(\infty)}}$ follows directly from \Cref{lemma:From_RN_to_Banach}.
	For the $\Gamma$-$\liminf$ inequality, we additionally use that $\range S_{\psi}$ is complete and therefore closed in $Z$.
\end{proof}

While the proof of equicoercivity (\Cref{thm:Equicoercivity_for_product_measures}) only uses the inequality $I_{\mu^{(n)}} \geq \logden_{\gamma^{(n)},m^{(n)}}^{(n)}$, which holds by \Cref{thm:OM_for_product_measures}, the $\mathsf{\Gamma}$-convergence of the corresponding \ac{OM} functionals relies on the complete knowledge of the \ac{OM} functionals which are \emph{assumed} to be given by $I_{\mu^{(n)}} = \logden_{\gamma^{(n)},m^{(n)}}^{(n)}$.
This assumption is proven in \Cref{thm:OM_for_product_measures} only for certain product measures.
For example, \Cref{thm:OM_for_product_measures} \change{applies to Cauchy measures and  Besov-$p$ measures with $p\in[1,2]$ (cf.\ \Cref{thm:OM_functional_for_Besov_with_p_in_1_2,cor:OM_Cauchy}), but} does not apply for Besov-$p$ measures with $p>2$, because $m + \ell_{\gamma}^{2} \subsetneqq E_{\gamma,m}$ in this case.
Therefore, \Cref{conjecture:OM_for_product_measures} remains an important open problem.

%% file: sec-04-convergence-03.tex
\subsection{Application to Besov measures}
\label{sec:Gamma_Besov}

This section considers the $\Gamma$-convergence of \ac{OM} functionals of Besov measures as introduced by \citet{LassasSaksmanSiltanen2009} and \citet{DashtiHarrisStuart2012}.\footnote{We are slightly more general in that we consider shifted Besov measures.}
We will consider Besov $B_{p}^{s}$ measures with integrability parameter $1 \leq p \leq 2$ and smoothness $s \in \Reals$, in contrast to the analysis of Part~I of this paper \citep[Sections~5.1 and 5.2]{AyanbayevKlebanovLieSullivan2021_I}, which was limited to the cases $p \in \{ 1, 2 \}$.

Throughout this subsection, we make use of the following notation:
\begin{notation}
	\label{notation:General_assumption_Besov}
	Let $s \in \Reals$, $d\in\Naturals$, $1\leq p \leq 2$, $\eta > 0$, $t \defeq s - p^{-1} d (1+\eta)$ and assume that
	$\tau \defeq (s/d + 1/2)^{-1} > 0$.
	Define $\gamma_{0} \defeq 1$ and $\gamma,\delta \in \Reals^{\Naturals}$ by
	\[
	\gamma_{k}
	\defeq
	k^{-\tfrac{1}{\tau} + \tfrac{1}{p}},
	\qquad
	\delta_{k}
	\defeq
	k^{-\tfrac{1}{\tau} + \tfrac{2 + \eta}{p}},
	\qquad
	k\in\Naturals,
	\]
	as well as the probability measures $\mu_{k}$, $k \in\Naturals \cup \{0\}$, on $\Reals$ with probability densities
	\[
	\frac{\rd \mu_{k}}{\rd u} (u)
	=
	\frac{1}{2\gamma_{k} \Gamma (1 + 1/p)} \, \exp\biggl(-\Absval{ \frac{u - m_{k}}{\gamma_{k}}}^p\biggr),
	\]
	where $m \in \ell_{\delta}^{p}$ is some fixed shift.
	Further, let $Z^{0}$ be a separable Hilbert space\footnote{Typically, Besov measures are introduced on the space $Z^{0} = L^2(\Torus^{d})$ with an orthonormal wavelet basis $\psi$ of sufficient regularity, in which case $X_{p}^{s}$ coincides with the Besov space $B^{s}_{p p} (\Torus^{d})$ --- as defined by \citet{Triebel1983} --- and $X_{2}^{s}$ coincides with the Sobolev space $H^{s} (\Torus^{d})$.
	In our more general definition, the dimension $d$ becomes superfluous and one could work with $\tilde{s} \defeq s / d$, but we continue to use the classical notation in order to reduce confusion.}
	with complete orthonormal basis $\psi = (\psi_{k})_{k \in \Naturals}$ and $\overline{S}_{\psi}\colon \Reals^{\Naturals} \to \prod_{k \in \Naturals}\spn{\psi_{k}}$, $c \mapsto \sum_{k \in \Naturals} c_{k} \psi_{k}$.
	We emphasise that the direct product $\prod_{k\in\Naturals}\spn{\psi_{k}}$ is neither $\spn \psi$ nor $Z^{0}$. In \Cref{cor:Q_space_Bps}, we state how $\overline{S}_{\psi}$ here is related to the \change{synthesis} operator $S_{\psi}:X \to Z$ from \Cref{notation:From_RN_to_Banach}.
\end{notation}

The role of $\eta,t$ and $\delta$ will be explained in \Cref{rem:Support_of_Besov}, where we discuss normed spaces of full Besov measure.
We define (shifted) Besov measures as follows, using notation that is an adaptation of that of \citet{DashtiHarrisStuart2012}:

\begin{definition}[sequence space Besov measures and Besov spaces]
	\label{def:Besov_space_and_measure_sequence}
	Using \Cref{notation:General_assumption_Besov}, we call $\mu \defeq \bigotimes_{k \in \Naturals} \mu_{k}$ a (\defterm{sequence space}) \defterm{Besov measure} on $\Reals^{\Naturals}$ and write $B^{s}_{p} \defeq B^{s,m,d}_{p}\defeq \mu$.
	The corresponding \defterm{Besov space} is the weighted sequence space $(X^{s}_{p},\norm{\quark}_{X^{s}_{p}}) \defeq (\ell^{p}_{\gamma},\norm{\quark}_{\ell^{p}_{\gamma}})$.
\end{definition}

\begin{definition}[Hilbert space Besov measures and Besov spaces]
	\label{def:Besov_space_and_measure_Hilbert}
	Using \Cref{notation:General_assumption_Besov}, if $\rv{v}_{k} \sim \mu_{k}$ are independent random variables, then we call $\rv{u} \defeq \sum_{k \in \Naturals} \rv{v}_{k} \psi_{k}$ a \defterm{Besov-distributed random variable} and its law a \defterm{Besov measure}, denoted by $B^{s}_{p}(\psi) \defeq B^{s,m,d}_{p}(\psi)$.
	Furthermore, let
	\begin{align*}
		\tilde{X}^{s}_{p}
		& \defeq
		\overline{S}_{\psi} (\ell_{\gamma}^{p}),
		&
		\Norm{\overline{S}_{\psi}(c)}_{X^{s}_{p}}
		& \defeq
		\norm{c}_{\ell_{\gamma}^{p}},
		\quad
		c\in \ell_{\gamma}^{p},
	\end{align*}
	and define the \defterm{Besov space} $X^{s}_{p} =  X^{s}_{p}(\psi)$ as the completion of $\tilde{X}^{s}_{p}$ with respect to $\norm{ \quark }_{X^{s}_{p}}$.
	By Parseval's identity, the initial space $Z^{0}$ coincides with the Besov space $X^{0}_{2}$.
\end{definition}

\begin{remark}
	Since it is the parameter $p$ that most strongly affects the qualitative properties of the measure, we often refer simply to a ``Besov-$p$ measure'' for any measure in the above class, regardless of the values of $s$, $d$, etc.
	The scaling of the Besov-2 measure corresponds to the ``physicist's Gaussian distribution'' rather than the ``probabilist's Gaussian distribution''.
	In particular, for $p = 2$, $\rv{v}_{k} \sim \mu_{k}$ has variance $\frac{1}{2} \gamma_{k}^{2}$.
	A consequence of this is that the \ac{OM} functional of the Besov-$p$ measure will be $\norm{ \quark }_{X^{s}_{p}}^{p}$, i.e.\ appears to lack a prefactor of $\frac{1}{p}$ relative to the Gaussian \ac{OM} functional --- one half of the square of the Cameron--Martin norm --- given by \citet[Section~5.1]{AyanbayevKlebanovLieSullivan2021_I}. 
\end{remark}

\begin{remark}
	\label{rem:Support_of_Besov}
	Note that the random variable $\rv{u} = \sum_{k \in \Naturals} \rv{v}_{k} \psi_{k}$ in \Cref{def:Besov_space_and_measure_Hilbert} takes values in a space $Z$ that may be larger than $Z^{0}$.
	It has already been shown by \citet[Lemma~2]{LassasSaksmanSiltanen2009} that, for $\tilde{t} \in \Reals$,
	\begin{equation*}
		\norm{ \rv{u} - \overline{S}_{\psi}(m) }_{X^{\tilde{t}}_{p}} < \infty \text{ a.s.}
		\iff
		\mathbb{E} \bigl[ \exp ( \alpha \norm{ \rv{u} - \overline{S}_{\psi}(m) }_{X^{\tilde{t}}_{p}}^{p} ) \bigr] < \infty \text{ for all $\alpha \in (0, \tfrac{1}{2}]$}
		\iff
		\tilde{t} < s - \frac{d}{p} .
	\end{equation*}
	Hence, using the choice $t \defeq s - p^{-1} d (1+\eta)$ in \Cref{notation:General_assumption_Besov}, $Z$ can be chosen as the Besov space $X^{t}_{p}(\psi) = \overline{S}_{\psi}(\ell_{\delta}^{p})$, i.e.\ ``just a bit larger than'' $X^{s - d / p}_{p}(\psi) = \overline{S}_{\psi}(\ell_{\gamma}^{p})$.
	The shift by $m \in \ell_{\delta}^{p}$ does not cause problems, since $\overline{S}_{\psi}(m)\in X^{t}_{p}(\psi)$.
	For the sequence space Besov measure $\mu = B^{s}_{p}$ on $\Reals^{\Naturals}$, the space $X_{p}^{t}  = \ell_{\delta}^{p}$ has full $\mu$-measure.
\end{remark}

	Given \Cref{rem:Support_of_Besov}, we will from now on consider the Besov measures $\mu = B^{s}_{p}$ and $\mu = B^{s}_{p}(\psi)$ as measures on the normed spaces $X=X_{p}^{t}$ and $Z=X_{p}^{t}(\psi)$, respectively.

Apart from the different degree of summability ($2$ in place of $p$), the next result can be interpreted as saying that the shifts $h$ with respect to which the $B_{p}^{s}$ measure is quasi-invariant are $\tfrac{d}{2}$ degrees smoother than the typical draws from that measure.
For $p=1$, the corresponding result was obtained in \citet[Lemma~3.5]{AgapiouBurgerDashtiHelin2018}, without using Shepp's theorem.

\change{In preparation for the next two results, we recall \Cref{notation:From_RN_to_Banach}: $X = \ell_{\alpha}^{p}$ for some $1 \leq p < \infty$ and $\alpha \in \Reals_{> 0}^{\Naturals}$, $Z$ is a separable Banach space with Schauder basis $\psi = (\psi_{k})_{k\in\Naturals}$, the synthesis operator $	S_{\psi}  \colon X \to Z$ satisfies $x = (x_{k})_{k \in \Naturals} \mapsto \sum_{k \in \Naturals} x_{k}\psi_{k}$, and the coordinate operator $T_{\psi}  \colon Z \to \Reals^{\Naturals}$ satisfies $z = \sum_{k \in \Naturals} v_{k} \psi_{k} \mapsto (v_{k})_{k\in\Naturals}$.	If $\mu \in \prob{X}$, then $\mu_{\psi} \defeq (S_{\psi})_{\#} \mu$ is the push-forward of $\mu$ under $S_{\psi}$. For the following result, $  X^{t}_{p}(\psi)$ and $B^{s}_{p}(\psi)$ are given in \Cref{def:Besov_space_and_measure_Hilbert}.} 
\begin{corollary}[Shift-quasi-invariance space and shift density of a Besov measure]
\label{cor:Q_space_Bps}
	Let $\mu = B^{s}_{p}$ be the sequence space Besov measure on $\Reals^{\Naturals}$ or on $X = X_{p}^{t} = \ell_{\delta}^{p}$.
	Then $Q(\mu) = \ell_{\gamma}^{2} = X_{2}^{s + \frac{d}{2} - \frac{d}{p}}$ and, for any $h \in Q(\mu)$ and $x\in\Reals^\Naturals$ (respectively $x\in X$),
	\begin{equation}
		\label{eq:Radon_nikodym_derivative_of_mu_h_wrt_mu_Besov}
		r_{h}^{\mu}(x)
		=
		\exp\left( \sum_{k \in \Naturals} \gamma_{k}^{-p} \bigl( \absval{x_{k} - m_{k}}^{p} - \absval{x_{k} - m_{k} -h_{k}}^{p} \bigr) \right).
	\end{equation}
	Further, using \Cref{notation:From_RN_to_Banach} with $\alpha = \delta$ and $Z = X^{t}_{p}(\psi)  = \overline{S}_{\psi}(\ell_{\delta}^{p})$, we have $S_{\psi} = \overline{S}_{\psi}|_{\ell_{\delta}^{p}}$ and $\mu_{\psi} = B^{s}_{p}(\psi)$.
	Then $Q(\mu_{\psi})= S_{\psi}(\ell_{\gamma}^{2}) = X_{2}^{s + \frac{d}{2} - \frac{d}{p}} (\psi)$ and, for any $h \in Q(\mu_{\psi})$ and $z\in Z$, $r_{h}^{\mu_{\psi}}(z) = r_{T_{\psi}(h)}^{\mu}(T_{\psi}(z))$.
\end{corollary}

\begin{proof}
\change{\Cref{assump:basic_assumptions_for_product_measures} \ref{item:basic_assumption_product_mu_0} and \ref{item:basic_assumption_product_scaled_marginals}, which concern the continuity and symmetry of the reference density $\rho$ and the assumption that each $\mu_{k}$ is related to $\mu_{0}$ by an affine transformation respectively, are satisfied by virtue of \Cref{def:Besov_space_and_measure_sequence}. \Cref{assump:basic_assumptions_for_product_measures}\ref{item:basic_assumption_product_full_measure}, which concerns the assumption that $X=\ell_{\alpha}^{p}$ and $\mu(X)=1$, follows from \Cref{rem:Support_of_Besov}, while \ref{item:basic_assumption_product_Shepp_condition}, which states that the reference density $\rho$ has finite Fisher information, follows from a straightforward computation.  \Cref{thm:shift_quasi_invariance_space_for_product_measures} yields the formula \eqref{eq:Radon_nikodym_derivative_of_mu_h_wrt_mu_Besov} for $r_{h}^{\mu}$, the spaces $Q(\mu)$, $Q(\mu_{\psi})$, and the equation for $r_{h}^{\mu_{\psi}}$.}
\end{proof}

The following corollary is an application of \Cref{thm:OM_for_product_measures,thm:Equicoercivity_for_product_measures,thm:Gamma_convergence_for_product_measures} to Besov-$p$ measures $\mu,\mu^{(n)}$, $n\in\Naturals$, $1 \leq p \leq 2$, with different smoothness parameters $s,s^{(n)}$ and shifts $m,m^{(n)}$ such that $s^{(n)} \to s$ and $m^{(n)} \to m$ as $n\to\infty$.
Note that it is not entirely clear on which space $X$ to consider equicoercivity and $\Gamma$-convergence, since the measures $\mu,\mu^{(n)}$ seem to live on different spaces $X = \ell_{\delta}^{p}$, $X^{(n)} = \ell_{\delta^{(n)}}^{p}$ with
\[
\delta_{k} = k^{-\frac{s}{d} - \frac{1}{2} + \frac{2+\eta}{p}},
\qquad
\delta_{k}^{(n)} = k^{-\frac{s^{(n)}}{d} - \frac{1}{2} + \frac{2+\eta^{(n)}}{p}},
\qquad
\eta, \eta^{(n)} > 0,
\qquad
n\in\Naturals.
\]
After all, \Cref{thm:Equicoercivity_for_product_measures,thm:Gamma_convergence_for_product_measures} explicitly demand all measures $\mu, \mu^{(n)}$ to be defined on the \emph{same} space $X = \ell_{\alpha}^{p}$.
However, as we will see, the assumed convergence $s^{(n)} \to s$ guarantees the existence of such a common space $X$ of full $\mu^{(n)}$-measure for all but finitely many $n\in \Naturals$.

\change{In preparation for the following corollary, we recall formula \eqref{equ:OM_from_RN_to_Banach}:
\begin{equation*}
		I_{\mu_{\psi}} \colon Z \to \eReals,
		\qquad
		I_{\mu_{\psi}} (h)
		=
		\begin{cases}
			I_{\mu}(T_{\psi}h) & \text{if } h \in \range S_{\psi},
			\\
			+\infty & \text{otherwise.}
		\end{cases}
	\end{equation*}
}

\begin{corollary}[\ac{OM} functional, equicoercivity and $\mathsf{\Gamma}$-convergence for Besov-p measure, $1 \leq p \leq 2$]
	\label{thm:OM_functional_for_Besov_with_p_in_1_2}
	Using \Cref{notation:General_assumption_Besov}, the \ac{OM} functional $I_{\mu} \colon X\to\eReals$ of $\mu = B_{p}^{s} = B^{s,m,d}_{p}$ on $X = X_{p}^{t} = \ell^p_\delta$ is given by
	\begin{equation}
		\label{equ:OM_Besov_p_1_2}
		I_{\mu}(h)
		=
		\begin{cases}
			\norm{h-m}_{X^{s}_{p}}^{p} = \norm{h-m}_{\ell^{p}_{\gamma}}^p
			&
			\text{if } h-m\in X^{s}_{p} = \ell^p_\gamma,
			\\
			\infty
			&
			\text{otherwise.}
		\end{cases}
	\end{equation}
	Further, let $\mu^{(n)} \defeq B^{s^{(n)}}_p = B^{s^{(n)},m^{(n)},d}_{p}$, $n\in\Naturals$, be Besov measures such that $s^{(n)}\to s$, $\norm{m^{(n)} - m}_{X}\to 0$ as $n\to\infty$ and $\frac{1}{\tau}-\frac{1}{p} = \frac{s}{d}+\frac{1}{2}-\frac{1}{p}>0$.	
	Then there exists $n_{0}\in\Naturals$ such that, for each $n\geq n_{0}$, $\mu^{(n)}(X) = 1$ and we therefore consider these measures on the same space $X = X_{p}^{t}  = \ell_{\delta}^{p}$.	
	Then the sequence $(I_{\mu^{(n)}})_{n \geq n_{0}}$ of \ac{OM} functionals of $\mu^{(n)}$ given by $I_{\mu^{(n)}} = \norm{\quark - m^{(n)}}_{X^{s^{(n)}}_{p}}^{p} \colon X \to \eReals$ is equicoercive and $I_{\mu} = \Gammalim_{n \to \infty} I_{\mu^{(n)}}$.	
	Similarly, using \Cref{notation:From_RN_to_Banach} and assuming $S_{\psi}$ to be an isometry, $I_{\mu_{\psi}}$ and $I_{\mu_{\psi}^{(n)}},\ n\in\Naturals$, defined by \eqref{equ:OM_from_RN_to_Banach} constitute \ac{OM} functionals for $\mu_{\psi} = B^{s,m,d}_{p}(\psi)$ and $\mu_{\psi}^{(n)} = B^{s^{(n)},m^{(n)},d}_{p}(\psi)$, respectively, and $(I_{\mu_{\psi}^{(n)}})_{n \geq n_{0}}$ is equicoercive with $I_{\mu_{\psi}^{(n)}} \xrightarrow[n \to \infty]{\Gamma} I_{\mu_{\psi}^{(\infty)}}$.	
\end{corollary}

\begin{proof}
	\change{\Cref{assump:basic_assumptions_for_product_measures} \ref{item:basic_assumption_product_full_measure}--\ref{item:basic_assumption_product_scaled_marginals} and 	\ref{item:basic_assumption_product_Besov_p_1_2} --- i.e.\ the support condition on $\mu$, continuity and symmetry of the reference density $\rho$, affine transformation property and Besov property --- }are satisfied by \Cref{def:Besov_space_and_measure_sequence,rem:Support_of_Besov} with
	\[
	\logden(u) = \absval{u}^{p},
	\qquad
	\logden_{\gamma,m}(h) = \norm{h-m}_{\ell^{p}_{\gamma}}^p,
	\qquad
	E_{\gamma,m} = m+\ell^{p}_{\gamma} \subseteq m+\ell_{\gamma}^{2},
	\]
	hence \eqref{equ:OM_Besov_p_1_2} follows directly from \Cref{thm:OM_for_product_measures}.
	In other words, the result in \Cref{conjecture:OM_for_product_measures} holds for the Besov measures $\mu$ and $\mu^{(n)}$:
	\change{\begin{align*}
		\lim_{r \searrow 0} \frac{\mu( \cBall{h}{r} )}{\mu( \cBall{m}{r})}
		&=
		\begin{cases}
			\exp\left( - \logden_{\gamma,m}(h)  \right) & \text{if } h \in E_{\gamma,m},
			\\
			0 & \text{if } h \notin E_{\gamma,m},
		\end{cases}
    \end{align*}
    and a similar result holds with $\mu$ replaced by $\mu^{(n)}$.}
	The analogous statement for  $I_{\mu_{\psi}}$ and $I_{\mu_{\psi}^{(n)}},\ n\in\Naturals$, follows from \Cref{lemma:From_RN_to_Banach}. \change{Recall that this lemma transforms an \ac{OM} functional on the sequence space $X$ into an \ac{OM} functional on the separable Banach space $Z$, where $X$ and $Z$ are related by the synthesis operator $S_{\psi}:X\to Z$ and coordinate operator $T_{\Psi}:Z\to \Reals^{\Naturals}$.}

	Since $s^{(n)} \to s$, there exists $n_{0} \in \Naturals$ such that, for $n\geq n_{0}$, $\absval{s^{(n)}-s} \leq \frac{d\eta}{2p}$.	
	Therefore, for  $n\geq n_{0}$, $t = s - p^{-1} d (1+\eta) < s^{(n)} - p^{-1}d$ and $\mu^{(n)}(X) = 1$ for $X = X_{p}^{t}  = \ell_{\delta}^{p}$ by \Cref{rem:Support_of_Besov}.	
	Further, for $n\geq n_{0}$, the sequences $a^{(n)} = (k^{-1-\eta} \, \absval{k^{\frac{p}{d}(s-s^{(n)})} - 1})_{k\in \Naturals}$ are (uniformly) bounded by the summable sequence $a = (2k^{-1-\eta/2})_{k\in\Naturals}$ and the reverse Fatou lemma implies
	\[
	\limsup_{n \to \infty}
	\norm{\gamma^{(n)} - \gamma}_{\ell_{\delta}^{p}}^{p}
	=
	\limsup_{n \to \infty}
	\sum_{k \in \Naturals} k^{-1-\eta} \, \Absval{k^{\frac{p}{d}(s-s^{(n)})} - 1}
	\leq
	\sum_{k \in \Naturals} \limsup_{n \to \infty} k^{-1-\eta} \, \Absval{k^{\frac{p}{d}(s-s^{(n)})} - 1}
	=
	0,
	\]
	proving $\norm{\gamma^{(n)} - \gamma}_{X}\to 0$.
	Equicoercivity and $\Gamma$-convergence of the sequences $(I_{\mu^{(n)}})_{n\in\Naturals}$ and $(I_{\mu_{\psi}^{(n)}})_{n\in\Naturals}$ \change{directly} follow from \Cref{thm:Equicoercivity_for_product_measures,thm:Gamma_convergence_for_product_measures} \change{respectively}.
\end{proof}

%% file: sec-04-convergence-04.tex
\subsection{Application to Cauchy measures}
\label{sec:Gamma_Cauchy}

This section considers infinite-dimensional Cauchy measures in the sense of infinite products of one-dimensional Cauchy distributions, as used by e.g.\ \citet{Sullivan2017} and \citet{LieSullivan2018Cauchy}.
We note that there is another class of ``Cauchy measures'' in the literature, namely the class of stochastic processes with Cauchy-distributed increments, as used by e.g.\ \citet{MarkkanenRoininenHuttunenLasanen2019} and \citet{ChadaRoininenSuuronen2021}.

\begin{definition}
	\label{def:Cauchy_distribution}
	We define the \defterm{Cauchy measure} $\Cauchy(m,\gamma) \defeq \bigotimes_{k \in \Naturals} \Cauchy(m_{k}, \gamma_{k})$ on $\Reals^{\Naturals}$ with shift parameter $m \in \Reals^{\Naturals}$ and scale parameter $\gamma \in \Reals_{>0}^{\Naturals}$ as the product measure of one-dimensional Cauchy measures on $\Reals$ with shift parameter $m_{k}$ and scale parameter $\gamma_{k}$, $k\in\Naturals$, i.e.\ with probability densities
	\[
		\frac{\rd \Cauchy(m_{k}, \gamma_{k})}{\rd u} (u)
		\defeq
		\biggl(
		\pi\gamma_{k} \biggl(1 + \Absval{\frac{u - m_{k}}{\gamma_{k}}}^{2} \biggr)
		\biggr)^{-1}=\frac{1}{\pi\gamma_{k}}\frac{\gamma_{k}^{2}}{\gamma_{k}^{2}+\absval{u-m_k}^{2}}.
	\]
\end{definition}

\begin{assumption}
	\label{assump:Cauchy_technical}
	$X = \ell^{q}$ for some $q \geq 1$, $m \in \ell^{q}$, $\gamma \in \ell^{1}(\Naturals) \cap \Reals_{>0}^{\Naturals}$.
	In addition, if $q=1$, then $\gamma$ satisfies $\sum_{k\in\Naturals}\absval{\gamma_{k}\log\absval{\gamma_{k}}}<\infty$.
\end{assumption}

\change{Recall \Cref{notation:From_RN_to_Banach}: $X = \ell_{\alpha}^{p}$ for some $1 \leq p < \infty$ and $\alpha \in \Reals_{> 0}^{\Naturals}$, $Z$ is a separable Banach space with Schauder basis $\psi = (\psi_{k})_{k\in\Naturals}$, the synthesis operator $	S_{\psi}  \colon X \to Z$ satisfies $x = (x_{k})_{k \in \Naturals} \mapsto \sum_{k \in \Naturals} x_{k}\psi_{k}$, and the coordinate operator $T_{\psi}  \colon Z \to \Reals^{\Naturals}$ satisfies $z = \sum_{k \in \Naturals} v_{k} \psi_{k} \mapsto (v_{k})_{k\in\Naturals}$.	If $\mu \in \prob{X}$, then $\mu_{\psi} \defeq (S_{\psi})_{\#} \mu$ is the push-forward of $\mu$ under $S_{\psi}$.}

\begin{definition}[{\citealt[Definition 3.2, Assumption 3.3]{Sullivan2017}}]
	\label{def:Cauchy_distributed_RV}
	Under \Cref{assump:Cauchy_technical} and using \Cref{notation:From_RN_to_Banach} with $\alpha \equiv 1$, we call $\rv{u} \defeq S_{\psi}(\rv{v}) = \sum_{k} \rv{v}_{k} \psi_{k}$, where $\rv{v} \sim \Cauchy(m,\gamma)$, a \defterm{Cauchy-distributed random variable in $Z$} and write $\rv{u} \sim \Cauchy^{q,\psi}(m,\gamma)$.	
	In other words, $\Cauchy^{q,\psi}(m,\gamma) = \mu_{\psi}$ for $\mu = \Cauchy(m,\gamma)$.
\end{definition}

The following theorem guarantees the well-definedness of the random variable $\rv{u}$ above:

\begin{theorem}[{\citealt[Theorem 3.4]{Sullivan2017}}]
	\label{thm:Cauchy_distributed_RV}
	Under \Cref{assump:Cauchy_technical}, the Cauchy measure $\mu = \Cauchy(m,\gamma)$ on $\Reals^{\Naturals}$ from \Cref{def:Cauchy_distribution} satisfies $\mu(X) = 1$.
	Similarly, under the assumptions of \Cref{def:Cauchy_distributed_RV}, $\rv{u} \in Z$ a.s.
\end{theorem}

\begin{lemma}
\label{lemma:Cauchy_fulfills_product_measure_assumptions}
The Cauchy measure $\mu = \Cauchy(m,\gamma)$ on $\Reals^{\Naturals}$ satisfies \Cref{assump:basic_assumptions_for_product_measures} \ref{item:basic_assumption_product_mu_0}--\ref{item:basic_assumption_product_integrable_second_derivative}.
Further, under \Cref{assump:Cauchy_technical}, \ref{item:basic_assumption_product_full_measure} is fulfilled for $X = \ell^{q}$.
\end{lemma}

\begin{proof}
\change{The support condition \ref{item:basic_assumption_product_full_measure} follows from \Cref{thm:Cauchy_distributed_RV}; the continuity and symmetry of the reference density $\rho$ \ref{item:basic_assumption_product_mu_0} and the affine transformation property of the $(\mu_{k})_{k\in\Naturals}$  \ref{item:basic_assumption_product_scaled_marginals} follow from \Cref{def:Cauchy_distribution}; the finite Fisher information \ref{item:basic_assumption_product_Shepp_condition} and smoothness assumptions on the reference density $\rho$ \ref{item:basic_assumption_product_integrable_second_derivative}} can be verified by straightforward computations.
\end{proof}

The following theorem characterises the shift-quasi-invariance space $Q(\mu)$ of the Cauchy measure $\mu = \Cauchy(m,\gamma)$ as well as the corresponding shift density $r_{h}^{\mu}$:

\begin{corollary}[Shift-quasi-invariance space and shift density of a Cauchy measure]
	\label{cor:Cauchy_shift_quasi_invariance_space}
	If $\mu = \Cauchy(m,\gamma)$ is the Cauchy measure on $\Reals^{\Naturals}$, then $Q(\mu) = \ell_{\gamma}^{2}$.
	In particular, if $\gamma \in \ell^{1}$, then $Q(\mu) \subseteq \ell^{2/3} \subseteq \ell^{1}$.
	In addition, for any $h \in Q(\mu)$ and $x\in\Reals^\Naturals$,
	\begin{equation}
		\label{eq:Radon_nikodym_derivative_of_mu_h_wrt_mu_Cauchy_series}
		r_{h}^{\mu}(x)
		=
		\lim_{N\to\infty}\prod_{n=1}^{N}\frac{(x_{k}-m_{k})^2+\gamma_{k}^2}{(x_{k}-m_{k}-h_{k})^2+\gamma_{k}^2}.
	\end{equation}
	Further, under \Cref{assump:Cauchy_technical} and using \Cref{notation:From_RN_to_Banach} with $\alpha \equiv 1$, we have $\mu_{\psi} = \Cauchy^{q,\psi}(m,\gamma)$.
	Then $Q(\mu_{\psi})= S_{\psi}(\ell_{\gamma}^{2})$ and, for any $h \in Q(\mu_{\psi})$ and $z\in Z$, $r_{h}^{\mu_{\psi}}(z) = r_{T_{\psi}(h)}^{\mu}(T_{\psi}(z))$.
\end{corollary}

\begin{proof}
	\Cref{assump:basic_assumptions_for_product_measures} \ref{item:basic_assumption_product_full_measure}--\ref{item:basic_assumption_product_Shepp_condition} are satisfied by \Cref{lemma:Cauchy_fulfills_product_measure_assumptions}. \change{\Cref{thm:shift_quasi_invariance_space_for_product_measures} yields the formula \eqref{eq:Radon_nikodym_derivative_of_mu_h_wrt_mu_Besov} for $r_{h}^{\mu}$, the spaces $Q(\mu)$, $Q(\mu_{\psi})$, and the equation for $r_{h}^{\mu_{\psi}}$. 
	\Cref{prop:embedding_of_ell_p_alpha_spaces} yields the containment relation $Q(\mu)\subseteq \ell^{2/3}$. }
\end{proof}

\begin{corollary}[\ac{OM} functional, equicoercivity and $\mathsf{\Gamma}$-convergence for Cauchy measure]
	\label{cor:OM_Cauchy}
	Under \Cref{assump:Cauchy_technical}, an \ac{OM} functional $I_{\mu} \colon X\to\eReals$ of $\mu=\Cauchy(m,\gamma)$ is given by
	\begin{equation*}
		I_{\mu}(h)
		=
		\begin{cases}
			\sum_{k\in\Naturals}\log \bigl( 1 + \gamma_{k}^{-2} (h_{k} - m_{k})^{2} \bigr)
			&
			\text{if } h\in m + \ell_{\gamma}^{2},
			\\
			\infty
			&
			\text{otherwise.}
		\end{cases}
	\end{equation*}
	Further, for $n\in\Naturals$, let $\mu^{(n)} = C(m^{(n)},\gamma^{(n)})$ be Cauchy measures
	such that $m^{(n)}$ and $\gamma^{(n)}$ satisfy \Cref{assump:Cauchy_technical} for the same $q\geq 1$ as above and
	$\norm{m^{(n)} - m}_{X}\to 0$ and $\norm{\gamma^{(n)} - \gamma}_{X}\to 0$ as $n\to\infty$.
	Then the sequence $(I_{\mu^{(n)}})_{n\in\Naturals}$ is equicoercive and $I_{\mu} = \Gammalim_{n \to \infty} I_{\mu^{(n)}}$.
	Similarly, using \Cref{notation:From_RN_to_Banach} with $\alpha \equiv 1$ and assuming $S_{\psi}$ to be an isometry, $I_{\mu_{\psi}}$ and $I_{\mu_{\psi}^{(n)}},\ n\in\Naturals$, defined by \eqref{equ:OM_from_RN_to_Banach} constitute \ac{OM} functionals for $\mu_{\psi} = \Cauchy^{q,\psi}(m,\gamma)$ and $\mu_{\psi}^{(n)} = \Cauchy^{q,\psi}(m^{(n)},\gamma^{(n)})$, respectively, and $(I_{\mu_{\psi}^{(n)}})_{n\in\Naturals}$ is equicoercive with $I_{\mu_{\psi}^{(n)}} \xrightarrow[n \to \infty]{\Gamma} I_{\mu_{\psi}^{(\infty)}}$.
\end{corollary}

\begin{proof}
	\Cref{assump:basic_assumptions_for_product_measures} \ref{item:basic_assumption_product_full_measure}--\ref{item:basic_assumption_product_integrable_second_derivative} are satisfied by \Cref{lemma:Cauchy_fulfills_product_measure_assumptions}.
	We have
	\[
	\logden(u) = \log(1+u^{2}),
	\qquad
	\logden_{\gamma,m}(h) = \sum_{k\in\Naturals}\log \bigl( 1 + \gamma_{k}^{-2} (h_{k} - m_{k})^{2} \bigr),
	\qquad
	E_{\gamma,m} = m + \ell_{\gamma}^{2},
	\]
	where we used that $\sum_{k\in\Naturals}\log \bigl( 1 + \gamma_{k}^{-2} (h_{k} - m_{k})^{2} \bigr)$ is finite if and only if $h-m \in \ell_{\gamma}^{2}$, as well as \Cref{cor:handy_inclusion_of_weighted_lp_spaces_general} to guarantee that $\ell_{\gamma}^{2} \subseteq X$.
	Thus, the first statement follows from \Cref{thm:OM_for_product_measures}, i.e.\ the result in \Cref{conjecture:OM_for_product_measures} holds for the Cauchy measures $\mu$ and $\mu^{(n)}$, $n\in\Naturals$:
	\change{\begin{align*}
		\lim_{r \searrow 0} \frac{\mu( \cBall{h}{r} )}{\mu( \cBall{m}{r})}
		&=
		\begin{cases}
			\exp\left( - \logden_{\gamma,m}(h)  \right) & \text{if } h \in E_{\gamma,m},
			\\
			0 & \text{if } h \notin E_{\gamma,m},
		\end{cases}
    \end{align*}
    and a similar result holds with $\mu$ replaced by $\mu^{(n)}$.}
	The analogous statement for  $I_{\mu_{\psi}}$ and $I_{\mu_{\psi}^{(n)}},\ n\in\Naturals$, follows from \Cref{lemma:From_RN_to_Banach}. \change{Recall that this lemma shows that an \ac{OM} functional on the sequence space $X$ yields an \ac{OM} functional on the separable Banach space $Z$, where $X$ and $Z$ are related by the synthesis operator $S_{\psi}:X\to Z$.}
	\change{The} equicoercivity and $\Gamma$-convergence of the sequences $(I_{\mu^{(n)}})_{n\in\Naturals}$ and $(I_{\mu_{\psi}^{(n)}})_{n\in\Naturals}$ \change{now} follow \change{directly} from \Cref{thm:Equicoercivity_for_product_measures,thm:Gamma_convergence_for_product_measures} \change{respectively}.
\end{proof}

%% file: sec-05-closing.tex
\section{Closing remarks}
\label{sec:closing}

In this paper, our first main contribution is to obtain a formula for the \ac{OM} functionals of a class of probability measures on a weighted sequence space $X=\ell_{\alpha}^{p}$.
This class is defined using \Cref{assump:basic_assumptions_for_product_measures}, and the key result that we used to obtain these formulas is \Cref{thm:OM_for_product_measures}.
In addition, we considered collections of measures in this class that converge to a limiting measure in the sense that the collections of shift and scale sequences converge to a limiting pair of shift and scale sequences, and convergence of the Lebesgue densities of the associated reference measures.
Our second main contribution is to state sufficient conditions for equicoercivity and $\Gamma$-convergence of the corresponding sequence of \ac{OM} functionals.
For this, we relied on \Cref{thm:Equicoercivity_for_product_measures} and \Cref{thm:Gamma_convergence_for_product_measures}.
In addition, we applied these results to Cauchy and Besov-$p$ measures for $1\leq p \leq 2$.
We used the results in the weighted sequence space setting to prove the analogous results for measures on separable Banach or Hilbert spaces.

In the context of \ac{BIP}s, the Besov, Cauchy, and more general product measures considered in this paper arise most naturally as prior distributions.
The results of this paper therefore provide a convergence theory for the corresponding prior \ac{OM} functionals.
Since these priors are unimodal, this convergence theory would appear to be surplus to requirements;
it is in some sense ``obvious'' how the modes of sequences of such measures ought to converge.
However, the importance of this paper's results is that prior $\Gamma$-convergence and equicoercivity can be transferred to the posterior using the results of Part~I of this paper \citep[Section~6]{AyanbayevKlebanovLieSullivan2021_I}, and understanding the convergence of posterior modes (i.e.\ \ac{MAP} estimators) is a non-trivial and novel contribution.

An important open problem raised in this paper is \Cref{conjecture:OM_for_product_measures}.
Proving this conjecture would significantly enhance the applicability of our results.
In addition, it would be of interest to study equicoercivity and $\Gamma$-convergence of so-called ``generalised \ac{OM} functionals'' as introduced by \citet{Clason2019GeneralizedMI}.

%% file: app-01-equivalence.tex
\section{Equivalence of product measures}
\label{sec:equivalence}

The following two dichotomies on the equivalence or mutual singularity of certain infinite product measures are classical results.
Here, $H(\mu,\nu)$ denotes the Hellinger integral defined in \eqref{eq:Hellinger_integral}.

\begin{theorem}[\citealp{Kakutani1948}]
	\label{thm:kakutanis_theorem}
	Let $(\mu_{k})_{k \in \Naturals}$ and $(\nu_{k})_{k \in \Naturals}$ be sequences in $\prob{\Reals}$ such that $\mu_{k} \sim \nu_{k}$ for all $k \in \Naturals$, and let $\mu \defeq \bigotimes_{k \in \Naturals} \mu_{k}$ and $\nu \defeq \bigotimes_{k \in \Naturals} \nu_{k}$.
	Then precisely one of the following alternatives holds true:
	\begin{enumerate}[label=(\alph*)]
		\item $H(\mu, \nu) = \prod_{k \in \Naturals} H(\mu_{k}, \nu_{k}) > 0$ and $\mu \sim \nu$, with density
		\begin{equation}
			\frac{\rd \nu}{\rd \mu} (u) = \lim_{K \to \infty} \prod_{k = 1}^{K} \frac{\rd \nu_{k}}{\rd \mu_{k}} (u_{k}) \quad \text{in } L^{1} (\Reals^{\Naturals}, \mu).
		\end{equation}
		\item $H(\mu, \nu) = \prod_{k \in \Naturals} H(\mu_{k}, \nu_{k}) = 0$ and $\mu \perp \nu$.
	\end{enumerate}
\end{theorem}

\begin{theorem}[{\citealp{Shepp1965}}]
	\label{thm:InvariantSpaceShepp}
	Let $\mu_{0} \in \prob{\Reals}$ have
	Lebesgue probability density $\rho$
	that satisfies \Cref{assump:basic_assumptions_for_product_measures} \ref{item:basic_assumption_product_Shepp_condition}.
	Further, let $h = (h_{k})_{k\in\Naturals} \in \Reals^{\Naturals}$, $\mu \defeq \bigotimes_{k\in\Naturals} \mu_{0}$, and $\nu \defeq \bigotimes_{k\in\Naturals} \mu_{0} (\quark - h_{k})$.
	Then precisely one of the following alternatives holds true:
	\begin{enumerate}[label=(\alph*)]
		\item
		$\sum_{k \in \Naturals} h_{k}^{2} < \infty$ and $\mu \sim \nu$.
		\item
		$\sum_{k \in \Naturals} h_{k}^{2} = \infty$ and $\mu \perp \nu$.
	\end{enumerate}
\end{theorem}

%% file: app-02-technical.tex
\section{Technical supporting results}
\label{sec:technical}

\begin{lemma}
	\label{lemma:Sigma_algebras_are_fine}
	Let $X = \ell_{\alpha}^{p}$ for some $1 \leq p < \infty$ and $\alpha \in \Reals_{> 0}^{\Naturals}$ and let $Y=\Reals^{\Naturals}$ be equipped with the product topology and the corresponding Borel $\sigma$-algebra $\Borel{Y}$.
	Then $\Borel{X} \subseteq \Borel{Y}$.
\end{lemma}

\begin{proof}
	By definition of the product topology, for $i\in\Naturals$, the projections $\pi_{i}(y) = y_{i}$, $y\in Y$, are continuous and so are the functions $f_{i}(y) = \absval{\frac{y_{i}-z_{i}}{\alpha_{i}}}^{p}$, where $z\in Y$ is any fixed sequence.
	Hence, the $(f_{i})_{i\in\Naturals}$ are Borel measurable, and so is the function $f(y) = \norm{y-z}_{\ell_{\alpha}^{p}}^{p}$ as a countable sum of non-negative measurable functions.
	Therefore each open ball $\cBall{z}{r} = f^{-1}( (-\infty,r^p) )$ lies in $\Borel{Y}$, and we have shown that $\Borel{X} \subseteq \Borel{Y}$.
\end{proof}

\begin{remark}
	In fact, $\Borel{X} = \{ B \cap X \mid B \in \Borel{Y}\}$.
	This can seen by considering sets of the form $\pi_{i}^{-1}( (a,b) ) \cap \ell_{\alpha}^{p}$, $a,b\in \Reals$.
	The collection of these sets forms a generator of $\Borel{\Reals^{\Naturals}}\cap\ell^p_\alpha$.
	The sets belong to $\Borel{\ell^p_\alpha}$, since they are open in $\ell^p_\alpha$.
\end{remark}

\begin{lemma}
	\label{lemma:Connection_between_scaling_sequences}
	Let \Cref{assump:basic_assumptions_for_product_measures} \ref{item:basic_assumption_product_full_measure}--\ref{item:basic_assumption_product_scaled_marginals} hold.
	Then:
	\begin{enumerate}[label = (\alph*)]
		\item
		\label{item:gamma_square_summable}
		$\gamma \in \ell_{\alpha}^{p}$.
		\item
		\label{item:gamma_tau_summable}
		$\gamma \in \ell_{\alpha}^{\tau}$ for some $0 < \tau < \infty $, if the following condition is fulfilled:
		\begin{equation}
		\label{equ:Decay_Rate_Density}
		\exists C>0\, \exists x_{0}>0\colon
		\quad
		x\geq x_{0}
		\ \implies\
		\int_{x}^{\infty} \rho(y)\, \rd y \geq C x^{-\tau}.
		\end{equation}
	\end{enumerate}
\end{lemma}

\begin{proof}
	Let $m=0$ and $\rv{v} = (\rv{v}_{k})_{k\in\Naturals} \sim \mu$, i.e.\ $\rv{v}_{k} = \gamma_{k} \rv{u}_{k}$ with $\rv{u}_{k} \stackrel{\text{i.i.d.}}{\sim} \mu_{0}$.
	Note that we may assume $m=0$ without loss of generality since $m\in X$ and therefore $\rv{v} \in X$ if and only if $\rv{v}+m \in X$.
	Let $\rv{w}_{k} \defeq \absval{\frac{\gamma_{k}}{\alpha_{k}} \rv{u}_{k}}^{p}$.
	Since $\mu(\ell_{\alpha}^{p}) = 1$, $\norm{ \rv{v} }_{\ell_{\alpha}^{p}}^{p} = \sum_{k \in \Naturals} \rv{w}_{k} < \infty$ a.s., which, by \citet[Theorem~5.18]{kallenberg2021foundations}, implies:
	\begin{enumerate}[label = (\roman*)]
		\item
		\label{item:Convergence_Of_Random_Series_Criterion_1}
		for any $A > 0$, $\sum_{k \in \Naturals} \bP( \absval{ \rv{w}_{k}} > A ) < \infty $ and
		\item
		\label{item:Convergence_Of_Random_Series_Criterion_2}
		$\sum_{k \in \Naturals} \bE \bigl[ \rv{w}_{k} \, \one_{\{ \absval{ \rv{w}_{k}} \leq 1 \}} \bigr] < \infty$.
	\end{enumerate}
	First note that \ref{item:Convergence_Of_Random_Series_Criterion_1} implies $\gamma_{k}/\alpha_{k} \to 0$ as $k\to\infty$.
	Hence, $c \defeq \min_{k\in\Naturals} c_{k}$ is strictly positive, where
	\[
	c_{k}
	\defeq
	\int_{-\alpha_{k}/\gamma_{k}}^{\alpha_{k}/\gamma_{k}} \absval{y}^{p}\, \rho(y) \, \rd y.
	\]
	Since $\absval{ \rv{w}_{k} } < 1$ if and only if $\absval{ \rv{u}_{k} } < \tfrac{\alpha_{k}}{\gamma_{k}}$, it follows from \ref{item:Convergence_Of_Random_Series_Criterion_2} that
	\[
	\infty
	>
	\sum_{k \in \Naturals} \bE \bigl[ \rv{w}_{k} \, \one_{\{ \absval{ \rv{w}_{k} } \leq 1 \}} \bigr]
	\geq
	\sum_{k \in \Naturals} \int_{-\alpha_{k}/\gamma_{k}}^{\alpha_{k}/\gamma_{k}}
	\absval{\tfrac{\gamma_{k}}{\alpha_{k}} y}^{p} \, \rho(y)\, \rd y
	=
	\sum_{k \in \Naturals} c_{k} \, \absval{\tfrac{\gamma_{k}}{\alpha_{k}}}^{p}
	\geq
	c \sum_{k \in \Naturals} \absval{\tfrac{\gamma_{k}}{\alpha_{k}}}^{p},
	\]
	proving \ref{item:gamma_square_summable}.
	If condition \eqref{equ:Decay_Rate_Density} is fulfilled, then there exists $K \in \Naturals$ such that, for all $k \geq K$, $\tfrac{\alpha_{k}}{\gamma_{k}} \geq x_{0}$ and thereby
	\[
	\bP[ \absval{ \rv{w}_{k} } > 1]
	=
	\int_{1}^{\infty} \tfrac{\alpha_{k}}{\gamma_{k}} \, \rho(\tfrac{\alpha_{k}}{\gamma_{k}}y)\, \rd y
	=
	\int_{\alpha_{k}/\gamma_{k}}^{\infty} \rho(y) \, \rd y
	\geq
	C \absval{\tfrac{\gamma_{k}}{\alpha_{k}}}^{\tau}.
	\]
	Hence, condition \ref{item:Convergence_Of_Random_Series_Criterion_1} implies \ref{item:gamma_tau_summable}.
\end{proof}

\begin{proposition}
	\label{prop:embedding_of_ell_p_alpha_spaces}
	Let $p,q \in [1,\infty)$ and $\alpha,\gamma \in \Reals_{> 0}^{\Naturals}$.
	Then $\ell_{\gamma}^{q} \subseteq \ell_{\alpha}^{p}$,
	\begin{itemize}
		\item
		if $p<q$ and $\gamma \in \ell_{\alpha}^{\frac{qp}{q-p}}$ (in particular, if $p<q$ and $\gamma \in \ell_{\alpha}^{p} \subseteq \ell_{\alpha}^{\frac{qp}{q-p}}$); or
		\item
		if $p \geq q$ and $\gamma \in \ell_{\alpha}^{\infty}$.
	\end{itemize}
\end{proposition}

\begin{proof}
	Let $p<q$ and $h \in \ell_{\gamma}^{q}$.
	By H\"{o}lder's inequality,
	\[
		\sum_{k \in \Naturals} \Absval{\frac{h_{k}}{\alpha_{k}}}^{p}
		=
		\sum_{k \in \Naturals} \Absval{\frac{h_{k}}{\gamma_{k}}}^{p}\cdot \Absval{\frac{\gamma_{k}}{\alpha_{k}}}^{p}
		\leq
		\Norm{ \biggl( \Absval{\frac{h_{k}}{\gamma_{k}}}^{p} \biggr)_{k \in \Naturals} }_{\ell^{\frac{q}{p}}} \cdot
		\Norm{ \biggl( \Absval{\frac{\gamma_{k}}{\alpha_{k}}}^{p} \biggr)_{k \in \Naturals} }_{\ell^{\frac{q}{q-p}}}
		=
		\norm{h}_{\ell_{\gamma}^{q}}^{p} \cdot
		\norm{\gamma}_{\ell_{\alpha}^{\frac{qp}{q-p}}}^{p}
		<
		\infty.
	\]
	Now let $p \geq q$ and $h \in \ell_{\gamma}^{q}$.
	Then there exists some constant $M>0$ such that for all $k$, $\absval{h_{k}/\gamma_{k}} \leq M$.
	Hence,
	\[
		\sum_{k \in \Naturals} \Absval{\frac{h_{k}}{\alpha_{k}}}^{p}
		=
		\sum_{k \in \Naturals} \Absval{\frac{h_{k}}{\gamma_{k}}}^{q}\cdot \frac{\absval{h_{k}}^{p-q}\gamma_{k}^{q}}{\alpha_{k}^{p}}
		\leq
		M^{p-q} \sum_{k \in \Naturals} \Absval{\frac{h_{k}}{\gamma_{k}}}^{q}\cdot  \Absval{\frac{\gamma_{k}}{\alpha_{k}}}^{p}
		=
		\change{M^{p-q} \norm{h}_{\ell_{\gamma}^{q}}^{q} \cdot
		\norm{\gamma}_{\ell_{\alpha}^{\infty}}^{p}}
		<
		\infty.
	\]
\end{proof}

\begin{corollary}
	\label{cor:handy_inclusion_of_weighted_lp_spaces_general}
	Under \Cref{assump:basic_assumptions_for_product_measures} \ref{item:basic_assumption_product_full_measure}--\ref{item:basic_assumption_product_scaled_marginals}, $\ell_{\gamma}^{2} \subseteq \ell_{\alpha}^{p}$.
\end{corollary}

\begin{proof}
 Since $\gamma \in \ell_{\alpha}^{p} \subseteq \ell_{\alpha}^{\infty}$ by \Cref{lemma:Connection_between_scaling_sequences}, the claim follows directly by considering the first and second alternatives in \Cref{prop:embedding_of_ell_p_alpha_spaces} for the case where $p<2$ and $p\geq 2$ respectively.
\end{proof}

\subsection{Proof of \Cref{thm:OM_for_product_measures}}
\label{section:Proof_Theorem_OM_for_product_measures}

In this section we give the proof of \Cref{thm:OM_for_product_measures} which is technical and requires additional notation and lemmas:

\begin{definition}
	A non-negative function $f\colon \Reals^{d} \to \Reals_{\geq 0}$, $d\in\Naturals$, has the \defterm{symmetric decay property} if
	\begin{itemize}
		\item
		$d = 1$ and $f$ is symmetric, i.e.\ $f(x)=f(-x)$ for every $x \in \Reals$, and the restriction $f|_{\Reals_{\geq 0}}$ is monotonically decreasing;
		\item
		$d > 1$ and $f$ has the symmetric decay property ``along each coordinate'', i.e., for any $u \in \Reals^{d}$, the functions $f(\quark,u_{2},\dots,u_{d}),\, f(u_{1},\quark,u_{3},\dots,u_{d}), \dots,\, f(u_{1},\dots,u_{d-1},\quark)$ have the symmetric decay property.
	\end{itemize}
\end{definition}

\begin{lemma}
	\label{lemma:Decay_property_inheritance}
	Let $d\in\Naturals\setminus \{1\}$, let both $s\colon \Reals^{d-1} \to \Reals_{\geq 0}$ and $f\colon \Reals^{d} \to \Reals_{\geq 0}$ have the symmetric decay property and let $g\colon \Reals \to \Reals_{\geq 0}$.
	Then $h\colon \Reals^{d-1} \to \Reals_{\geq 0}$ also has the symmetric decay property, where
	\[
	h(u)
	\defeq
	\int_{-s(u)}^{s(u)} f(u,v) \, g(v) \, \rd v.
	\]
\end{lemma}

\begin{proof}
	We will show that $h$ has the symmetric decay property along the first coordinate.
	The proofs for the other coordinates proceed analogously.
	For any $u = (u_{2},\dots,u_{d-1})\in\Reals^{d-2}$ and any $u_{1}, u_{1}' \in \Reals$ with $\absval{u_{1}} \leq \absval{u_{1}'}$, it holds that $s(u_{1},u) \geq s(u_{1}',u)$, and therefore
	\[
	h(u_{1},u)
	=
	\int_{-s(u_{1},u)}^{s(u_{1},u)} f(u_{1},u,v) \, g(v) \, \rd v
	\geq
	\int_{-s(u_{1}',u)}^{s(u_{1}',u)} f(u_{1}',u,v) \, g(v) \, \rd v
	=
	h(u_{1}',u).
	\]
	The symmetry of $h$ follows directly from the symmetry of $s$ and $f$.
\end{proof}

\begin{lemma}
	\label{lemma:Inequalities_Volumes_around_h_1_d}
	Let $s>0$ and $f,g\colon [-s,s] \to \Reals$ both have the symmetric decay property and $v \in \Reals$.
	Then
	\[
	\int_{-s}^{s} f(u+v)\, g(u)\, \rd u
	\leq
	\int_{-s}^{s} f(u)\, g(u)\, \rd u.
	\]
\end{lemma}

\begin{proof}
	Due to symmetry, we only need to consider $v \geq 0$, and we split this into two cases, according to whether or not $v \leq 2 s$.

	We first consider the case that $v \in [0,2s]$.
	First note that $g(u+v) \leq g(u)$ for any $u \in [-\tfrac{v}{2} , s-v]$.
	For $u\geq 0$, this follows from the symmetric decay property.
	For $u \in [-\tfrac{v}{2} , 0]$, it holds that $u + v \geq \tfrac{v}{2}$, and thus $g(u+v) \leq g(\tfrac{v}{2}) = g(-\tfrac{v}{2}) \leq g(u)$.
	Using the transformation $u \mapsto -u-v$ we obtain
	\begin{align*}
	\int_{-s}^{-v/2} (f(u+v)-f(u))\, g(u)\, \rd u
	&=
	\int_{-v/2}^{s-v} (f(-u)-f(-u-v))\, g(-u-v)\, \rd u
	\\
	&=
	\int_{-v/2}^{s-v} (f(u)-f(u+v))\, g(u+v)\, \rd u
	\\
	&\leq
	-\int_{-v/2}^{s-v} (f(u+v)-f(u))\, g(u)\, \rd u.
	\end{align*}
	Further, for any $u \in [s-v,s]$, $u+v \geq s$, and thus $f(u+v) \leq f(s) \leq f(u)$.
	Therefore,
	{\small
		\begin{align*}
			& \int_{-s}^{s} (f(u+v)-f(u))\, g(u)\, \rd u
			\\
			& \quad =
			\underbrace{\int_{-s}^{-v/2} (f(u+v)-f(u))\, g(u)\, \rd u
				+
				\int_{-v/2}^{s-v} (f(u+v)-f(u))\, g(u)\, \rd u}_{\leq\, 0}
			+
			\int_{s-v}^{s} \underbrace{(f(u+v)-f(u))}_{\leq\, 0}\, \underbrace{g(u)}_{\geq\, 0}\, \rd u
			\\
			& \quad \leq
			0.
		\end{align*}
	}

	Secondly, we consider the case that $v > 2s$.
	For any $u \in [-s,s]$, $u+v > s$ and thus $f(u+v) \leq f(s) \leq f(u)$.
	Therefore,
	\[
		\int_{-s}^{s} \underbrace{(f(u+v)-f(u))}_{\leq\, 0}\, \underbrace{g(u)}_{\geq\, 0}\, \rd u
		\leq
		0.
	\]
\end{proof}

\begin{lemma}
	\label{lemma:Perturbation_by_shifted_rho}
	Under \Cref{assump:basic_assumptions_for_product_measures} \ref{item:basic_assumption_product_mu_0} and \ref{item:basic_assumption_product_integrable_second_derivative},
	there exists $M>0$ such that, for any $s > 0$, any $\Lambda\colon \Reals \to \Reals_{\geq 0}$ with the symmetric decay property, and any $v \in \Reals$ with $\absval{v}\leq 1$,
	\begin{enumerate}[label = (\alph*)]
		\item
		\label{item:Perturbation_shifted_second_derivative}
		$\displaystyle
		\Absval{\int_{-s}^{s} \rho''(u+v)\, \Lambda(u) \, \rd u}
		\leq
		M \int_{-s}^{s} \rho(u)\, \Lambda(u) \, \rd u;
		$
		\item
		\label{item:Perturbation_by_shifted_rho}
		there exists $\zeta = \zeta(s,\Lambda,v) \in [-\frac{M}{2},\frac{M}{2}]$ such that
		\[
		\int_{-s}^{s} \rho(u+v)\, \Lambda(u) \, \rd u
		=
		\bigl(1 + \zeta \, v^{2}\bigr)
		\int_{-s}^{s} \rho(u)\, \Lambda(u) \, \rd u.
		\]
	\end{enumerate}
\end{lemma}

\begin{proof}
	Since $\rho$ is a probability density and $\rho'' \in L^{1}(\Reals)$ by \ref{item:basic_assumption_product_integrable_second_derivative}, we can choose $s_{\ast} > 0$ such that
	\[
	\int_{S_{\ast}} \rho(u)\, \rd u
	\geq
	\frac{1}{2},
	\qquad
	\int_{\Reals \setminus \tilde{S}_{\ast}} \absval{\rho''(u)}\, \rd u
	\leq
	\frac{1}{2},
	\]
	where $S_{\ast} \defeq [-s_{\ast},s_{\ast}]$ and $\tilde{S}_{\ast} \defeq [-s_{\ast}-1,s_{\ast}+1]$.
	Hence, for any  $v \in \Reals$ with $\absval{v}\leq 1$, it follows that $\int_{\Reals \setminus S_{\ast}} \absval{\rho''(u+v)}\, \rd u \leq \frac{1}{2}$.
	Since $\tilde{S}_{\ast}$ is compact, $\rho$ and $\rho''$ are continuous and $\rho$ is strictly positive by \Cref{assump:basic_assumptions_for_product_measures} \ref{item:basic_assumption_product_mu_0}, there exists $M > 1$ such that, for any $u_{1},u_{2} \in \tilde{S}_{\ast}$,
	\[
	\Absval{\frac{\rho''(u_{1})}{\rho(u_{2})}} \leq M-1.
	\]
	Now let $s > 0$, $S\defeq [-s,s]$, $\Lambda\colon \Reals \to \Reals_{\geq 0}$ be any function with the symmetric decay property and $v \in \Reals$ with $\absval{v}\leq 1$.
	By the mean value theorem for definite integrals, there exists for any closed interval $A\subseteq S_{\ast}$ some $u_{A} = u_{A}(\Lambda,v) \in S_{\ast}$ such that
	\begin{equation}
	\label{equ:Technical_perturbation_second_derivative_inside}
	\Absval{\int_{A} \rho''(u+v)\, \Lambda(u) \, \rd u}
	\leq
	\Absval{\frac{\rho''(u_{A}+v)}{\rho(u_{A})}}
	\int_{A} \rho(u)\, \Lambda(u) \, \rd u
	\leq
	(M-1) \int_{A} \rho(u)\, \Lambda(u) \, \rd u.
	\end{equation}
	If $s\leq s_{\ast}$, then $S \subseteq S_{\ast}$ and the proof of \ref{item:Perturbation_shifted_second_derivative} is finished.
	Otherwise, since $\Lambda$ has the symmetric decay property,
	\begin{align}
	\begin{split}
	\label{equ:Technical_perturbation_second_derivative_outside}
	\Absval{\int_{S \setminus S_{\ast}} \rho''(u+v)\, \Lambda(u) \, \rd u}
	&\leq
	\Lambda(s_{\ast}) \int_{S \setminus S_{\ast}} \absval{\rho''(u+v)}\, \rd u
	\leq
	\Lambda(s_{\ast}) \frac{1}{2}
	\leq
	\Lambda(s_{\ast}) \int_{S_{\ast}} \rho(u)\, \rd u
	\\
	&\leq
	\int_{S_{\ast}} \rho(u)\, \Lambda(u) \, \rd u.
	\end{split}
	\end{align}
	Hence, combining \eqref{equ:Technical_perturbation_second_derivative_inside} and \eqref{equ:Technical_perturbation_second_derivative_outside} and using $S_{\ast} \subset S$,
	\begin{align*}
	\Absval{\int_{S} \rho''(u+v)\, \Lambda(u) \, \rd u}
	& \leq
	\Absval{\int_{S_{\ast}} \rho''(u+v)\, \Lambda(u) \, \rd u}
	+
	\Absval{\int_{S \setminus S_{\ast}} \rho''(u+v)\, \Lambda(u) \, \rd u} \\
	& \leq
	M \int_{S} \rho(u)\, \Lambda(u) \, \rd u,
	\end{align*}
	proving \ref{item:Perturbation_shifted_second_derivative}.
	Now let $F_{s}(t) \defeq \int_{-s}^{s} \rho(u+t)\, \Lambda(u)\, \rd u$.
	Since $\rho \in C^{2}(\Reals)$,
	\[
	F_{s}'(t) \defeq \int_{-s}^{s} \rho'(u+t)\, \Lambda(u)\, \rd u,
	\qquad
	F_{s}''(t) \defeq \int_{-s}^{s} \rho''(u+t)\, \Lambda(u)\, \rd u.
	\]
	By symmetry of $\rho$ and $\Lambda$, $F_{s}'(0) = 0$ and, if $\absval{t}\leq 1$,
	\ref{item:Perturbation_shifted_second_derivative} implies
	\[
	\absval{F_{s}''(t)}
	\leq
	M
	\int_{-s}^{s} \rho(u)\, \Lambda(u)\, \rd u
	=
	M F_{s}(0).
	\]
	Hence, by Taylor's theorem, there exists $\xi \in [-v,v] \subseteq [-1,1]$ and $\zeta = \zeta(s,\Lambda,v) \in [-\frac{M}{2},\frac{M}{2}]$ such that
	\[
	F_{s}(v)
	=
	F_{s}(0) + F_{s}'(0)\, v + F_{s}''(\xi)\, \frac{v^{2}}{2}
	=
	F_{s}(0) \, ( 1 + \zeta\, v^{2} ),
	\]
	proving \ref{item:Perturbation_by_shifted_rho}.
\end{proof}

\begin{lemma}
	\label{lem:Perturbation_by_shifts_for_Besov}
	For $p\in [1,2]$, $s > 0$, any symmetric function $\Lambda\colon \Reals \to \Reals_{\geq 0}$ and any $v \in \Reals$,
	\[
	\int^{s}_{-s} e^{-|u+v|^p}\, \Lambda(u)\, \rd u
	\geq
	e^{-|v|^p}\int^{s}_{-s} e^{-|u|^p}\, \Lambda(u)\, \rd u.
	\]
\end{lemma}

\begin{proof}
If $1 < p \leq 2$, then \citet[Theorem 2]{Clarkson} yields, for any $x,y \in \Reals$,
\[
\absval{x+y}^{p} + \absval{x-y}^{p}
\geq
2^{p-1}(\absval{x}^{p} + \absval{y}^{p}).
\]
Using the transformation $x = u+v$, $y = u-v$ proves
\begin{equation}
	\label{eq:parallelogram_law}
	2(\absval{u}^{p} + \absval{v}^{p}) \geq \absval{u+v}^{p} + \absval{u-v}^{p}
\end{equation}
for any $u,v\in \Reals$, whenever $1 < p \leq 2$, while for $p=1$ the inequality 	\eqref{eq:parallelogram_law} follows directly from the triangle inequality.
Using the inequality $e^{x} \geq 1+x$, $x\in\Reals$, it follows that
\[
e^{-|u+v|^p + |u|^p + |v|^p} + e^{-|u-v|^p + |u|^p + |v|^p}
\geq
2 - |u+v|^p - |u-v|^p + 2 |u|^p + 2 |v|^p
\geq
2
\]
and therefore
\[
e^{-|u+v|^p} + e^{-|u-v|^p} \geq 2 e^{- |u|^p - |v|^p}.
\]
Since $\Lambda\colon \Reals \to \Reals_{\geq 0}$ is even and non-negative, we obtain
\begin{align*}
	\int^{s}_{-s} e^{-|u+v|^p}\, \Lambda(u)\, \rd u
	&=
	\int^{s}_{0} \bigl(e^{-|u+v|^p} + e^{-|u-v|^p}\bigr) \, \Lambda(u)\, \rd u
	\\
	&\geq
	\int^{s}_{0} 2 e^{- |v|^p - |u|^p} \, \Lambda(u)\, \rd u
	\\
	&=
	e^{-|v|^p} \int^{s}_{-s} e^{-|u|^p} \, \Lambda(u)\, \rd u.
\end{align*}
\end{proof}

\change{
\begin{notation}
	\label{notation:OM_for_product_measures}
	Under \Cref{assump:basic_assumptions_for_product_measures} \ref{item:basic_assumption_product_full_measure}--\ref{item:basic_assumption_product_scaled_marginals}, we introduce the following notation for any $r > 0$ and any $a,b \in \Naturals$:
	\begin{itemize}
		\item
		$\displaystyle [a\mathord{:}b]
		\defeq
		\{ a,\dots,b \}$.
		\item
		For $x \in \Reals^{\Naturals}$ define $x_{[a:b]} \defeq (x_{i})_{i\in [a:b]}$.
		\item
		$\displaystyle
		\cylBall{x}{r}{[a:b]}
		\defeq
		\biggl\{ y\in\Reals^{[a:b]} \ \bigg| \ \norm{y-x}_{\ell_{\alpha}^{p}([a:b])} = \Big(\sum_{k\in [a:b]} \absval{\alpha_{k}^{-1}(y_{k}-x_{k})}^{p}\Big)^{1/p} < r  \biggr\}$ for $x\in \Reals^{[a:b]}$.
		
		\item
		$\displaystyle
		\cylBall{x|z}{r}{[1:a]}
		\defeq
		\cylBall{x}{r(z)}{[1:a]}$,
		where
		$r(z)
		\defeq
		\Bigl( r^{p} - \norm{z}_{\ell_{\alpha}^{p}([a+1:b])}^{p} \Bigr)^{1/p}$, for $x\in \Reals^{[1:a]}$ and $z \in \cylBall{0}{r}{[a+1:b]}$.

		\item
		$\lambda_{[a:b]}$ denotes the Lebesgue measure on $\Reals^{[a:b]}$.
		\item
		$\mu^{[a:b]} = \bigotimes_{k \in [a:b]} \mu_{k}$ is the probability measure on $(\Reals^{[a:b]},\Borel{\Reals^{[a:b]}})$ given by the density
		\[
		\rho^{[a:b]}(x)
		\defeq
		\prod_{k \in [a:b]} \gamma_{k}^{-1}\, \rho(\gamma_{k}^{-1}x_{k}),
		\qquad
		x\in\Reals^{[a:b]}.
		\]

		\item
		Let $\logden$, $\logden_{\gamma,m}$ and $E_{\gamma,m}$ be defined as in \Cref{def:logdensity}.
		Recall that $\logden$ is continuous and $\logden(0) = 0$.
		Thus, for any $\varepsilon > 0$ and $u \in \Reals$, there exists $\delta_{u}(\varepsilon) > 0$ such that
		\[
		\absval{v} < \delta_{u}(\varepsilon)
		\ \implies\
		\absval{\logden(u+v) - \logden(u)} \leq \varepsilon,
		\
		\absval{\logden(v)} \leq \varepsilon.
		\]
		\item
		$ \displaystyle
		V_{r}(h,a,b)
		\defeq
		\int_{\cylBall{0}{r}{[a+1:b]}} \rho^{[a+1:b]} (u + h_{[a+1:b]})
		\, \lambda_{[1:a]}(\cylBall{0|u}{r}{[1:a]})
		\, \rd u,
		\qquad
		h \in X.$
		\item
		$\displaystyle \gamma_{[a:b]} \odot u \defeq (\gamma_{k}u_{k})_{k\in[a:b]},
		\quad
		\gamma_{[a:b]}^{-1} \odot A \defeq \Set{u \in \Reals^{[a:b]}}{\gamma_{[a:b]} \odot u \in A},
		\quad
		u \in \Reals^{[a:b]},\, A \subseteq \Reals^{[a:b]};$

		\item
		For $u \in \gamma_{[a+1:b]}^{-1} \odot \cylBall{0}{r}{[a+1:b]}$, we define
		\begin{align*}
		s_{r}^{[a+1:k]}(u_{[a+1:k]})
		&\defeq
		\frac{\alpha_{k+1}}{\gamma_{k+1}} \biggl( r^{p} - \sum_{j=a+1}^{k} \Absval{\frac{\gamma_{j}u_{j}}{\alpha_{j}}}^{p} \biggr)^{1/p},
		&&
		a \leq k < b ,
		\\
		\Lambda_{b,r}^{[a+1:b]}(u_{[a+1:b]})
		&\defeq
		\lambda_{[1:a]}\bigl(\cylBall{0|\gamma_{[a+1:b]} \odot u}{r}{[1:a]}\bigr),
		\\
		\Lambda_{b,r}^{[a+1:k]}(u_{[a+1:k]})
		&\defeq
		\int_{-s_{r}^{[a+1:k]}(u_{[a+1:k]})}^{s_{r}^{[a+1:k]}(u_{[a+1:k]})}
		\rho(u_{k+1}) \Lambda_{b,r}^{[a+1:k+1]}(u_{[a+1:k+1]}) \,
		\rd u_{k+1},
		&& a \leq k < b.
		\end{align*}
		For $k = a$, we use the convention that the empty sum in the parentheses is zero.
		Hence, we define $s_{r}^{a} \defeq s_{r}^{[a+1:a]}(u_{[a+1:a]}) \defeq \frac{\alpha_{a+1}}{\gamma_{a+1}}\, r$ in this case.			
	\end{itemize}
\end{notation}
}

\begin{lemma}
	\label{lemma:Symmetric_decay_property_Lambda}
	For any $a \leq k < b$ and $u \in \gamma_{[a+1:b]}^{-1} \odot \cylBall{0}{r}{[a+1:b]}$,
	the functions $s_{r}^{[a+1:k]}$ and $\Lambda_{b,r}^{[a+1:k+1]}$ satisfy the symmetric decay property, where we extend them to the corresponding Euclidean space by setting them to zero outside their domain.
	Further, $\Lambda_{b,r}^{[a+1:a]} = V_{r}(0,a,b)$.
\end{lemma}

\begin{proof}
	The symmetric decay properties of $s_{r}^{[a+1:k]}$, $a \leq k < b$, and $\Lambda_{b,r}^{[a+1:b]}$ follow directly from the definitions.
	The symmetric decay property of $\Lambda_{b,r}^{[a+1:k]}$, $a < k < b$, then follows recursively by consecutive application of \Cref{lemma:Decay_property_inheritance} with $g=\rho$.
	The statement $\Lambda_{b,r}^{[a+1:a]} = V_{r}(0,a,b)$ follows from the definitions of $\rho^{[a:b]}$, $V_{r}(0,a,b)$, $s_{r}^{[a+1:k]}$ and $\Lambda_{b,r}^{[a+1:k]}$, $a \leq k \leq b$.
\end{proof}

\begin{lemma}
	\label{lemma:Inequalities_Volumes_around_h}
	Let \Cref{assump:basic_assumptions_for_product_measures} \ref{item:basic_assumption_product_full_measure}--\ref{item:basic_assumption_product_scaled_marginals} hold with $m=0$.
	Then, using \Cref{notation:OM_for_product_measures}, for any $r>0$, $a,b \in \Naturals$ and $h \in X$,
	\[
	\lambda_{[1:a]}(\cylBall{0}{r}{[1:a]})
	\geq
	V_{r}(m,a,b)=V_{r}(0,a,b)
	\geq
	V_{r}(h,a,b).
	\]
	If, in addition, either \Cref{assump:basic_assumptions_for_product_measures} \ref{item:basic_assumption_product_integrable_second_derivative}  or \ref{item:basic_assumption_product_Besov_p_1_2} is satisfied, then for any $h \in E_{\gamma,m} \cap \ell_{\gamma}^{2}$ with $\gamma_{k}^{-1} \absval{h_{k}} \leq 1$, $k\in[a+1:b]$,
	\[
	V_{r}(h,a,b)
	\geq
	V_{r}(0,a,b) \, \prod_{k\in\cJ} c_{k},
	\qquad
	c_{k}
	=
	\begin{cases}
		1 + \overline{\zeta}_{k}\, \Absval{\frac{h_{k}}{\gamma_{k}}}^{2}
		& \text{if \ref{item:basic_assumption_product_integrable_second_derivative} holds,}
		\\[1ex]
		\exp\left(- \Absval{\frac{h_{k}}{\gamma_{k}}}^{p}\right)
		& \text{if \ref{item:basic_assumption_product_Besov_p_1_2} holds,}
	\end{cases}
	\]
	for certain $\overline{\zeta}_{k} \in [-\frac{M}{2},\frac{M}{2}]$ with $M>1$ as in \Cref{lemma:Perturbation_by_shifted_rho}.
\end{lemma}

\begin{proof}
	Since $\rho^{[a+1:b]} (\quark + h_{[a+1:b]})$ integrates to $1$ as a probability density, and since for any $u \in \Reals^{[a+1:b]}$ it holds that $\cylBall{0|u}{r}{[1:a]} \subseteq \cylBall{0}{r}{[1:a]}$, the first inequality follows.
	Let $\tilde{h}_{k} \defeq \gamma_{k}^{-1} h_{k}$.
	The second inequality follows by applying \Cref{lemma:Inequalities_Volumes_around_h_1_d} and by using $\Lambda_{b,r}^{[a+1:a]} = V_{r}(0,a,b)$ (cf.\ \Cref{lemma:Symmetric_decay_property_Lambda}):
	
	\newpage
	
	\begin{align*}
		V_{r}(h,a,b)
		&=
		\int_{\cylBall{0}{r}{[a+1:b]}} \rho^{[a+1:b]} (u + h_{[a+1:b]})
		\, \lambda_{[1:a]}(\cylBall{0|u}{r}{[1:a]})
		\, \rd u
		\\
		&=
		\int_{\gamma_{[a+1:b]}^{-1} \odot \cylBall{0}{r}{[a+1:b]}}
		\biggl(
		\prod_{k = a+1}^{b} \rho (u_{k} + \tilde{h}_{k})
		\biggr)
		\, \lambda_{[1:a]}\bigl( \cylBall{0|\gamma_{[a+1:b]} \odot u}{r}{[1:a]} \bigr)
		\, \rd u
		\\
		&=
		\int_{-s_{r}^{a}}^{s_{r}^{a}}
		\rho (u_{a+1} + \tilde{h}_{a+1})
		\int_{-s_{r}^{[a+1:a+1]}(u_{a+1})}^{s_{r}^{[a+1:a+1]}(u_{a+1})}
		\rho (u_{a+2} + \tilde{h}_{a+2})
		\,
		\cdots
		\\  
		& \hspace{2em}
		\underbrace{
			\int_{-s_{r}^{[a+1:b-1]}(u_{[a+1:b-1]})}^{s_{r}^{[a+1:b-1]}(u_{[a+1:b-1]})}
			\rho (u_{b} + \tilde{h}_{b})
			\, \Lambda_{b,r}^{[a+1:b]}(u_{[a+1:b]})
			\, \rd u_{b}
		}_{\leq\, \Lambda_{b,r}^{[a+1:b-1]}(u_{[a+1:b-1]}) \text{ by \Cref{lemma:Inequalities_Volumes_around_h_1_d,lemma:Symmetric_decay_property_Lambda}}}
		\cdots \rd u_{a+2} \, \rd u_{a+1}
		\\
		&\leq
		\int_{-s_{r}^{a}}^{s_{r}^{a}}
		\rho (u_{a+1} + \tilde{h}_{a+1})
		\int_{-s_{r}^{[a+1:a+1]}(u_{a+1})}^{s_{r}^{[a+1:a+1]}(u_{a+1})}
		\rho (u_{a+2} + \tilde{h}_{a+2})
		\,
		\cdots
		\\
		& \hspace{2em}
		\underbrace{
			\int_{-s_{r}^{[a+1:b-2]}(u_{[a+1:b-2]})}^{s_{r}^{[a+1:b-2]}(u_{[a+1:b-2]})}
			\rho (u_{b-1} + \tilde{h}_{b-1})
			\, \Lambda_{b,r}^{[a+1:b-1]}(u_{[a+1:b-1]})
			\, \rd u_{b-1}
		}_{\leq\, 		\Lambda_{b,r}^{[a+1:b-2]}(u_{[a+1:b-2]}) \text{ by \Cref{lemma:Inequalities_Volumes_around_h_1_d,lemma:Symmetric_decay_property_Lambda}}
		}
		\cdots \rd u_{a+2} \, \rd u_{a+1}
		\\
		&\vdotsquark
		\\
		&\leq
		\Lambda_{b,r}^{[a+1:a]}
		=
		V_{r}(0,a,b).
	\end{align*}
	
	Now let \change{in addition} \Cref{assump:basic_assumptions_for_product_measures} \ref{item:basic_assumption_product_integrable_second_derivative}
	hold, $h \in E_{\gamma,m} \cap \ell_{\gamma}^{2}$ and $M>1$ as in \Cref{lemma:Perturbation_by_shifted_rho}.
	Then, for $k = a+1,\dots,b$, there exist values $\zeta_{k}(u_{[a+1:k-1]}) \in [-\frac{M}{2},\frac{M}{2}]$ by \Cref{lemma:Perturbation_by_shifted_rho} and $\overline{\zeta}_{k} \in [-\frac{M}{2},\frac{M}{2}]$ by the mean value theorem for definite integrals, such that
	\begin{align*}
		\allowdisplaybreaks
		V_{r}(h,a,b)
		&=
		\int_{\cylBall{0}{r}{[a+1:b]}} \rho^{[a+1:b]} (u + h_{[a+1:b]})
		\, \lambda_{[1:a]}(\cylBall{0|u}{r}{[1:a]})
		\, \rd u
		\\
		&=
		\int_{\gamma_{[a+1:b]}^{-1} \odot \cylBall{0}{r}{[a+1:b]}}
		\biggl(
		\prod_{k = a+1}^{b} \rho (u_{k} + \tilde{h}_{k})
		\biggr)
		\, \lambda_{[1:a]}\bigl( \cylBall{0|\gamma_{[a+1:b]} \odot u}{r}{[1:a]} \bigr)
		\, \rd u
		\\
		&=
		\int_{-s_{r}^{a}}^{s_{r}^{a}}
		\rho (u_{a+1} + \tilde{h}_{a+1})
		\int_{-s_{r}^{[a+1:a+1]}(u_{a+1})}^{s_{r}^{[a+1:a+1]}(u_{a+1})}
		\rho (u_{a+2} + \tilde{h}_{a+2})
		\,
		\cdots
		\\  
		& \hspace{2em}
		\underbrace{
			\int_{-s_{r}^{[a+1:b-1]}(u_{[a+1:b-1]})}^{s_{r}^{[a+1:b-1]}(u_{[a+1:b-1]})}
			\rho (u_{b} + \tilde{h}_{b})
			\, \Lambda_{b,r}^{[a+1:b]}(u_{[a+1:b]})
			\, \rd u_{b}
		}_{
			\geq \bigl( 1 + \zeta_{b}(u_{[a+1:b-1]}) \, \absval{\tilde{h}_{b}}^{2} \bigr)
			\Lambda_{b,r}^{[a+1:b-1]}(u_{[a+1:b-1]}) \text{ by \Cref{lemma:Perturbation_by_shifted_rho,lemma:Symmetric_decay_property_Lambda}}
		}
		\cdots \rd u_{a+2} \, \rd u_{a+1}
		\\
		& \quad \geq
		\bigl( 1 + \overline{\zeta}_{b}\, \absval{\tilde{h}_{b}}^{2} \bigr)
		\int_{-s_{r}^{a}}^{s_{r}^{a}}
		\rho (u_{a+1} + \tilde{h}_{a+1})
		\int_{-s_{r}^{[a+1:a+1]}(u_{a+1})}^{s_{r}^{[a+1:a+1]}(u_{a+1})}
		\rho (u_{a+2} + \tilde{h}_{a+2})
		\,
		\cdots
		\\
		& \hspace{2em}
		\underbrace{
			\int_{-s_{r}^{[a+1:b-2]}(u_{[a+1:b-2]})}^{s_{r}^{[a+1:b-2]}(u_{[a+1:b-2]})}
			\rho (u_{b-1} + \tilde{h}_{b-1})
			\, \Lambda_{b,r}^{[a+1:b-1]}(u_{[a+1:b-1]})
			\, \rd u_{b-1}
		}_{
			\geq \bigl( 1 + \zeta_{b-1}(u_{a+1},\dots,u_{b-2}) \, \tilde{h}_{b-1}^{2} \bigr)
			\Lambda_{b,r}^{[a+1:b-2]}(u_{[a+1:b-2]}) \text{ by \Cref{lemma:Perturbation_by_shifted_rho,lemma:Symmetric_decay_property_Lambda}}
		}
		\cdots \rd u_{a+2} \, \rd u_{a+1}
		\\
		\intertext{and iterating this process yields}
		& V_{r}(h,a,b)
		\geq
		\Lambda_{b,r}^{[a+1:a]} \, \prod_{k=a+1}^{b}
		\bigl( 1 + \overline{\zeta}_{k}\, \tilde{h}_{k}^{2} \bigr)
		=
		V_{r}(0,a,b) \, \prod_{k=a+1}^{b}
		\bigl( 1 + \overline{\zeta}_{k}\, \tilde{h}_{k}^{2} \bigr),
	\end{align*}
	proving the first formula for $c_{k}$.
	Now, let \Cref{assump:basic_assumptions_for_product_measures} \ref{item:basic_assumption_product_Besov_p_1_2} be satisfied instead of \ref{item:basic_assumption_product_integrable_second_derivative}.
	Then $\rho(u_{k}+\tilde{h}_{k})\propto \exp(-\absval{u_{k}+\tilde{h}_{k}}^p)$.
	Using \Cref{lem:Perturbation_by_shifts_for_Besov} instead of \Cref{lemma:Perturbation_by_shifted_rho} and replacing $\bigl( 1 + \zeta_{k}(u_{[a+1:k-1]}) \, \absval{\tilde{h}_{k}}^{2} \bigr)$ and $\bigl( 1 + \overline{\zeta}_{k}\, \tilde{h}_{k}^{2} \bigr)$ by $\exp \bigl(- \absval{\tilde{h}_{k}}^{p}\bigr)$ in the above derivation, we obtain the second formula for $c_{k}$.
	Note that, in the case that \ref{item:basic_assumption_product_integrable_second_derivative} holds, all $\geq$ inequalities in the above derivation are actually equalities.
	We stated them as inequalities such that the proof can be transferred to the case where \ref{item:basic_assumption_product_Besov_p_1_2} is satisfied.
\end{proof}

\begin{lemma}
	\label{lemma:Inequality_for_small_balls_using_Volumes_around_h}
	Under \Cref{assump:basic_assumptions_for_product_measures} \ref{item:basic_assumption_product_full_measure}--\ref{item:basic_assumption_product_scaled_marginals} and using \Cref{notation:OM_for_product_measures}, for any $r>0$, $a,b \in \Naturals$ and $h\in X$,
	\begin{align*}
		\mu^{[1:b]}( \cylBall{h_{[1:b]}}{r}{[1:b]} )
		&\geq
		V_{r}(h,a,b) \, \inf_{v \in \cylBall{0}{r}{[1:a]}} \rho^{[1:a]} (v+h_{[1:a]}),
		\\
		\mu^{[1:b]}( \cylBall{h_{[1:b]}}{r}{[1:b]} )
		&\leq
		V_{r}(h,a,b) \, \sup_{v \in \cylBall{0}{r}{[1:a]}} \rho^{[1:a]} (v+h_{[1:a]}).		
	\end{align*}
\end{lemma}

\begin{proof}
	For any $u \in \Reals^{[a+1:b]}$, $\cylBall{0|u}{r}{[1:a]} \subseteq \cylBall{0}{r}{[1:a]}$.
	\change{
	In addition, $\cylBall{0}{r}{[1:b]}=\biguplus_{u\in\cylBall{0}{r}{[a+1:b]}}\cylBall{0|u}{r}{[1:a]} \times \{u\}$, where $\uplus$ indicates a disjoint union.}
	\change{This is because every $y \in \cylBall{0}{r}{[1:b]}$ satisfies $y=(y_{[1:a]},y_{[a+1:b]})\in\Reals^{[1:b]}$, where $y_{[1:a]}\in \cylBall{0|y_{[a+1:b]}}{r}{[1:a]}$ and $y_{[a+1:b]}\in \cylBall{0}{r}{[a+1:b]}$ are unique.
	This partition of the domain of integration yields
	}

	\begin{align*}
		\mu^{[1:b]}( \cylBall{h_{[1:b]}}{r}{[1:b]} )
		&=
		\int_{\cylBall{0}{r}{[1:b]}} \rho^{[1:b]} (y + h_{[1:b]}) \, \rd y
		\\
		&=
		\int_{\cylBall{0}{r}{[a+1:b]}} \rho^{[a+1:b]} (u+h_{[a+1:b]})
		\int_{\cylBall{0|u}{r}{[1:a]}} \rho^{[1:a]} (v+h_{[1:a]})
		\, \rd v \, \rd u
		\\
		&\geq
		\int_{\cylBall{0}{r}{[a+1:b]}} \rho^{[a+1:b]} (u+h_{[a+1:b]})
		\, \lambda_{[1:a]}(\cylBall{0|u}{r}{[1:a]})
		\inf_{v \in \cylBall{0|u}{r}{[1:a]}} \rho^{[1:a]} (v+h_{[1:a]})
		\, \rd u
		\\
		&\geq
		V_{r}(h,a,b)\, \inf_{v \in \cylBall{0}{r}{[1:a]}} \rho^{[1:a]} (v+h_{[1:a]}).
	\end{align*}
	A similar argument yields the second inequality.
\end{proof}

\begin{proof}[Proof of \Cref{thm:OM_for_product_measures}]
	Since $\norm{\quark}_{X} = \norm{\quark}_{\ell_{\alpha}^{p}}$, we have that, for any $r>0$ and $h\in X$, $\cBall{h}{r} = \bigcap_{K \in \Naturals} \bigl( \cylBall{h}{r}{[1:K]} \times \Reals^{\Naturals\setminus [1:K]} \bigr)$.
	Thus, by the continuity of probability measures,
	\[
		\mu(\cBall{h}{r})
		=
		\lim_{K \to \infty}
		\mu\bigl( \cylBall{h_{[1:K]}}{r}{[1:K]} \times \Reals^{\Naturals\setminus [1:K] }\bigr)
		=
		\lim_{K \to \infty}
		\mu^{[1:K]}\bigl( \cylBall{h_{[1:K]}}{r}{[1:K]}\bigr)
	\]
	and thereby
	\[
		\frac{\mu(\cBall{h}{r})}{\mu(\cBall{m}{r})}
		=
		\lim_{K \to \infty} \frac{\mu^{[1:K]}( \cylBall{h_{[1:K]}}{r}{[1:K]} )}{\mu^{[1:K]}( \cylBall{m_{[1:K]}}{r}{[1:K]} )}.
	\]
    The proof will now be established using the following three steps, of which the third is straightforward:
	
	\smallskip

	\noindent\textbf{Step 1.} Let $m=0$.
	For every $h \in X$, $N>0$ and $0<\varepsilon<1$ there exist $r_{\ast}>0$ and $K_{\ast}\in\Naturals$ such that for any $0<r<r_{\ast}$ and $K>K_{\ast}$,
	\[
	- \log \frac{\mu^{[1:K]}( \cylBall{h_{[1:K]}}{r}{[1:K]} )}{\mu^{[1:K]}( \cylBall{m_{[1:K]}}{r}{[1:K]} )}
	\geq
	\begin{cases}
	(1-\varepsilon) \, \logden_{\gamma,m}(h) - \varepsilon & \text{if } h\in E_{\gamma,m},
	\\
	N & \text{if } h\notin E_{\gamma,m}.
	\end{cases}
	\]
	Since the right-hand side does not depend on $r$ and $K$ and since $N,\varepsilon > 0$ are arbitrary, this proves \eqref{eq:Limit_ball_ratio_product_measures_upper_bound} for  $m=0$.

	\smallskip

	\noindent\textbf{Step 2.}
	Let $m=0$.
	If either \Cref{assump:basic_assumptions_for_product_measures} \ref{item:basic_assumption_product_integrable_second_derivative} or \ref{item:basic_assumption_product_Besov_p_1_2} is satisfied, there exist, for every $h \in E_{\gamma,m} \cap \ell_{\gamma}^{2}$ and $0 < \varepsilon < 1$, values $r_{\ast}>0$ and $K_{\ast}\in\Naturals$ such that, for any $0<r<r_{\ast}$ and $K>K_{\ast}$,
	\[
	- \log \frac{\mu^{[1:K]}( \cylBall{h_{[1:K]}}{r}{[1:K]} )}{\mu^{[1:K]}( \cylBall{m_{[1:K]}}{r}{[1:K]} )}
	\leq
	\logden_{\gamma,m}(h) + \varepsilon.
	\]
	Since the right-hand side does not depend on $r$ and $K$ and since $\varepsilon > 0$ is arbitrary, this proves \eqref{eq:Limit_ball_ratio_product_measures_lower_bound_l_2} for $m=0$.

	\smallskip

	\noindent\textbf{Step 3.} For arbitrary $m \in X$, \eqref{eq:Limit_ball_ratio_product_measures_upper_bound} and \eqref{eq:Limit_ball_ratio_product_measures_lower_bound_l_2} follow directly from
	\[
	\logden_{\gamma,m}(\quark) = \logden_{\gamma,0}(\quark - m),
	\qquad
	E_{\gamma,m} = m + E_{\gamma,0},
	\]
	finalizing the proof.

	\medskip
	
	We now give the proofs of the non-trivial first and second steps.
	
	\medskip

	\noindent\textbf{Proof of Step 1.}
	Let $m=0$, $h \in X$, $N>0$ and $0<\varepsilon<1$ and denote $\tilde{h}_{k} \defeq \gamma_{k}^{-1} h_{k}$.
	Choose $K_{\ast}$ such that
	\begin{equation}
	\label{equ:technical_choice_K_star_in_Step_1}
	\sum_{k=1}^{K_{\ast}} \logden(\tilde{h}_{k})
	\geq
	\begin{cases}
		(1-\varepsilon) \, \logden_{\gamma,m}(h) & \text{if } h\in E_{\gamma,m},
		\\
		N + \varepsilon & \text{if } h\notin E_{\gamma,m},
	\end{cases}
	\end{equation}
	where $\logden_{\gamma,m}(h)=\sum_{k\in\Naturals}\logden(\tilde{h}_{k})$ by \eqref{eq:formal_negative_log-density} and the assumption that $m=0$.
    Recall the definition of $\delta_{u}(\varepsilon)$ in \Cref{notation:OM_for_product_measures}.
    Choose
	\[
	r_{\ast}
	\defeq
	\min_{k = 1,\dots,K_{\ast}} \frac{\gamma_{k}}{\alpha_{k}}\,
	\delta_{\tilde{h}_{k}} \bigl(\tfrac{\varepsilon}{2K_{\ast}}\bigr)
	>
	0,
	\]
	which implies the following inequalities for any $0<r\leq r_{\ast}$, $v\in \cylBall{0}{r}{[1:K_{\ast}]}$ and $k\in[1:K_{\ast}]$:
	\begin{equation}
	\label{equ:technical_equation_continuity_q}
	\absval{\gamma_{k}^{-1} v_{k}}
	\leq
	\delta_{\tilde{h}_{k}}\bigl(\tfrac{\varepsilon}{2K_{\ast}}\bigr),
	\qquad
	\logden(\gamma_{k}^{-1} v_{k} + \tilde{h}_{k})
	\geq
	\logden(\tilde{h}_{k}) - \tfrac{\varepsilon}{2K_{\ast}},
	\qquad
	\logden(\gamma_{k}^{-1} v_{k})
	\leq
	\tfrac{\varepsilon}{2K_{\ast}}.
	\end{equation}
	It follows for any $0<r\leq r_{\ast}$, $K \geq K_{\ast}$ that
	\begin{align*}	
	- \log & \frac{\mu^{[1:K]}( \cylBall{h_{[1:K]}}{r}{[1:K]} )}{\mu^{[1:K]}( \cylBall{0_{[1:K]}}{r}{[1:K]} )}
	\\
	&\geq
	- \log \frac{
		V_{r}(h,K_{\ast},K) \, \sup_{v \in \cylBall{0}{r}{[1:K_{\ast}]}} \rho^{[1:K_{\ast}]} (v+h_{[1:K_{\ast}]})
	}
	{
		V_{r}(0,K_{\ast},K) \, \inf_{v \in \cylBall{0}{r}{[1:K_{\ast}]}} \rho^{[1:K_{\ast}]} (v),
	}
	& & \text{by \Cref{lemma:Inequality_for_small_balls_using_Volumes_around_h}}
	\\
	&\geq
	\inf_{v \in \cylBall{0}{r}{[1:K_{\ast}]}} \sum_{k = 1}^{K_{\ast}} \logden(\gamma_{k}^{-1} (v_{k} + h_{k})) -
	\sup_{v \in \cylBall{0}{r}{[1:K_{\ast}]}} \sum_{k = 1}^{K_{\ast}} \logden(\gamma_{k}^{-1} v_{k})
	& & \text{by \Cref{lemma:Inequalities_Volumes_around_h}}
	\\
	&\geq
	\sum_{k = 1}^{K_{\ast}} \bigl(\logden(\tilde{h}_{k}) - \tfrac{\varepsilon}{2K_{\ast}} - \tfrac{\varepsilon}{2K_{\ast}} \bigr)
	& & \text{by \eqref{equ:technical_equation_continuity_q}}
	\\
	&\geq
	- \varepsilon  + \sum_{k=1}^{K_{\ast}} \logden\bigl(\tilde{h}_{k} \bigr)
	\\
	&\geq
	\begin{cases}
	(1-\varepsilon)\, \logden_{\gamma,m}(h) - \varepsilon & \text{if } h\in E_{\gamma,m},
	\\
	N & \text{if } h\notin E_{\gamma,m},
	\end{cases}
	& & \text{by \eqref{equ:technical_choice_K_star_in_Step_1}}.
	\end{align*}
	
	\smallskip

	\noindent\textbf{Proof of Step 2.}
	Let $m=0$, $h \in E_{\gamma,m} \cap \ell_{\gamma}^{2}$ and $0<\varepsilon<1$ and denote $\tilde{h} \defeq (\tilde{h}_{k})_{k \in \Naturals} \defeq (\gamma_{k}^{-1} h_{k})_{k \in \Naturals}$.

	First let the additional \Cref{assump:basic_assumptions_for_product_measures} \ref{item:basic_assumption_product_integrable_second_derivative}
	hold.
	Since $\tilde{h}\in \ell^{2}$, we can choose $K_{\ast} \in \Naturals$ such that $\sum_{k=K_{\ast} + 1}^{\infty} \tilde{h}_{k}^{2} < \frac{\varepsilon}{2M}$, where $M>1$ is chosen as in \Cref{lemma:Perturbation_by_shifted_rho}.
	In particular, $\absval{\tilde{h}_{k}} \leq 1$ for all $k > K_{\ast}$.
	Let $r>0$ and $K \geq K_{\ast}+1$ be arbitrary.
	It follows from the second conclusion of \Cref{lemma:Inequalities_Volumes_around_h} that
	\begin{equation}
	\label{equ:lower_bound_on_V_ratio}
	\log \frac{V_{r}(h,K_{\ast},K)}{V_{r}(0,K_{\ast},K)}
	\geq
	\sum_{k = K_{\ast}+1}^{K} \log\bigl( 1 - \tfrac{M}{2} \tilde{h}_{k}^{2} \bigr)
	\geq
	-M \sum_{k = K_{\ast}+1}^{K} \tilde{h}_{k}^{2}
	\geq
	-\frac{\varepsilon}{2},
	\end{equation}
	where we used that $0 \leq \tfrac{M}{2} \tilde{h}_{k}^{2} < \tfrac{1}{4}$ and $\log(1-x) \geq - \tfrac{x}{1-x} \geq -2x$
	for $0 \leq x \leq \tfrac{1}{2}$.

	Similarly, if \Cref{assump:basic_assumptions_for_product_measures} \ref{item:basic_assumption_product_Besov_p_1_2} holds in place of \ref{item:basic_assumption_product_integrable_second_derivative}, then $h\in E_{\gamma,m} = \ell_{\gamma}^{p}$ implies the existence of $K_{\ast} \in \Naturals$ such that $\sum_{k=K_{\ast} + 1}^{\infty} |\tilde{h}_{k}|^{p} < \varepsilon/2$.
	In particular, $\absval{\tilde{h}_{k}} \leq 1$ for all $k > K_{\ast}$.
	Again, for any $r>0$ and $K \geq K_{\ast} +1$ it follows from the second conclusion of \Cref{lemma:Inequalities_Volumes_around_h} that
	\begin{equation}
		\label{equ:lower_bound_on_V_ratio_for_Besov_1_2}
		\log \frac{V_{r}(h,K_{\ast},K)}{V_{r}(0,K_{\ast},K)}
		\geq
		-\sum_{k = K_{\ast}+1}^{K} |\tilde{h}_{k}|^{p}
		\geq
		-\frac{\varepsilon}{2}.
	\end{equation}
	The rest of the proof is identical for both \Cref{assump:basic_assumptions_for_product_measures} \ref{item:basic_assumption_product_integrable_second_derivative} and \ref{item:basic_assumption_product_Besov_p_1_2}.
	Recall the definition of $\delta_{u}(\varepsilon)$ in \Cref{notation:OM_for_product_measures} and choose
	\[
	r_{\ast}
	\defeq
	\min_{k = 1,\dots,K_{\ast}} \frac{\gamma_{k}}{\alpha_{k}}\,
	\delta_{\tilde{h}_{k}} \bigl(\tfrac{\varepsilon}{2K_{\ast}}\bigr)
	>
	0,
	\]
	which implies the following inequalities for any $0<r\leq r_{\ast}$, $v\in \cylBall{0}{r}{[1:K_{\ast}]}$ and $k\in[1:K_{\ast}]$:
	\begin{equation}
	\label{equ:technical_equation_continuity_q_Step_2}
	\absval{\gamma_{k}^{-1} v_{k}}
	\leq
	\delta_{\tilde{h}_{k}}\bigl(\tfrac{\varepsilon}{2K_{\ast}}\bigr),
	\qquad
	\logden(\gamma_{k}^{-1} v_{k} + \tilde{h}_{k})
	\leq
	\logden(\tilde{h}_{k}) + \tfrac{\varepsilon}{2K_{\ast}}.
	\end{equation}
	Since $\rho$ is symmetric and $\rho\vert_{\Reals_{\geq 0}}$ is monotonically decreasing, it follows that $\logden$ is symmetric and nonnegative on $\Reals$, and $\logden\vert_{\Reals_{\geq 0}}$ is monotonically increasing, with $\logden(0)=0$.
	It follows for any $0<r\leq r_{\ast}$ and $K \geq K_{\ast}$ that
	\begin{align*}	
	- \log & \frac{\mu^{[1:K]}( \cylBall{h_{[1:K]}}{r}{[1:K]} )}{\mu^{[1:K]}( \cylBall{0_{[1:K]}}{r}{[1:K]} )}
	\\
	&\leq
	- \log \frac{
		V_{r}(h,K_{\ast},K) \, \inf_{v \in \cylBall{0}{r}{[1:K_{\ast}]}} \rho^{[1:K_{\ast}]} (v+h_{[1:K_{\ast}]})
	}
	{
		V_{r}(0,K_{\ast},K) \, \sup_{v \in \cylBall{0}{r}{[1:K_{\ast}]}} \rho^{[1:K_{\ast}]} (v),
	}
	& & \text{by \Cref{lemma:Inequality_for_small_balls_using_Volumes_around_h}}
	\\
	&\leq
	\frac{\varepsilon}{2}
	+
	\sup_{v \in \cylBall{0}{r}{[1:K_{\ast}]}} \sum_{k = 1}^{K_{\ast}} \logden(\gamma_{k}^{-1} (v_{k} + h_{k})) -
	\underbrace{
		\inf_{v \in \cylBall{0}{r}{[1:K_{\ast}]}} \sum_{k = 1}^{K_{\ast}} \logden(\gamma_{k}^{-1} v_{k})
	}_{= \, 0}
	& & \text{by \eqref{equ:lower_bound_on_V_ratio} and \eqref{equ:lower_bound_on_V_ratio_for_Besov_1_2}}
	\\
	&\leq
	\frac{\varepsilon}{2} +
	\sum_{k =1}^{K_{\ast}} \bigl( \logden\bigl(\tilde{h}_{k} \bigr) + \tfrac{\varepsilon}{2K_{\ast}} \bigr)
	& & \text{by \eqref{equ:technical_equation_continuity_q_Step_2}}
	\\
	&\leq
	\logden_{\gamma,m}(h) + \varepsilon,
	\end{align*}
	where $\inf_{v \in \cylBall{0}{r}{[1:K_{\ast}]}} \sum_{k = 1}^{K_{\ast}} \logden(\gamma_{k}^{-1} v_{k})
	= 0$ follows from the nonnegativity of $\logden$ on $\Reals$.
\end{proof}

%% file: bit-acknowledgements.tex
\section*{Acknowledgements}
\addcontentsline{toc}{section}{Acknowledgements}

BA and TJS are supported in part by the \ac{DFG} through project \href{https://gepris.dfg.de/gepris/projekt/415980428}{415980428}.
Portions of this work were completed during the employment of BA and TJS at the Freie Universit\"at Berlin and while guests of the Zuse Institute Berlin, and during the employment of IK at the Zuse Institute Berlin.
IK and TJS have been supported in part by the \ac{DFG} through projects TrU-2 and EF1-10 of the Berlin Mathematics Research Centre MATH+ (EXC-2046/1, project \href{https://gepris.dfg.de/gepris/projekt/390685689}{390685689}).
\change{The research of HCL has been partially funded by the \ac{DFG} --- Project-ID \href{https://gepris.dfg.de/gepris/projekt/318763901}{318763901} --- SFB1294.}
\change{The authors thank two anonymous peer reviewers for their helpful suggestions.}